\pdfoutput=1
\documentclass[english]{jnsao}
\usepackage[utf8]{inputenc}

\usepackage[nameinlink,capitalize]{cleveref}
\usepackage{amsmath,amsfonts}
\usepackage{mathrsfs}
\usepackage{cancel}
\usepackage{changes}

\usepackage{enumitem}
\setlist[enumerate]{label=(\alph*)}
\makeatletter
\newcommand*{\greeknum}[1]{\expandafter\@greeknum\csname c@#1\endcsname}
\newcommand*{\@greeknum}[1]{\ifcase#1\or\alpha\or\beta\or\gamma\or\delta\else\@ctrerr\fi}
\makeatother
\AddEnumerateCounter{\greeknum}{\@greeknum}{$\beta$}
\usepackage{pdfrender}

\usepackage{marginnote}

\allowdisplaybreaks

\numberwithin{equation}{section}

\newcommand{\keywordsName}{Keywords}
\newcommand{\keywordsAnd}{~$\cdot$~}
\newcommand{\subclassName}{MSC (2020)}
\newcommand{\subclassAnd}{~$\cdot$~}

\newenvironment{keywords}{%
\vspace*{-0.75cm}
    \small%
    \quotation%
    \noindent%
    {\normalfont\sectfont\nobreak\keywordsName\quad}%
    }{%
    \endquotation%
    \vskip 0.7cm
}{}
\newenvironment{subclass}{%
	\vspace*{-0.75cm}
    \small%
    \quotation%
    \noindent%
    {\normalfont\sectfont\nobreak\subclassName\quad}%
    }{%
    \endquotation%
    \vskip 0.7cm
}{}

\newcommand{\mscLink}[1]{\href{http://www.ams.org/mathscinet/msc/msc2020.html?t=#1}{#1}}

\newcommand\norm[1]{\left\Vert#1\right\Vert}
\newcommand\nnorm[1]{\Vert#1\Vert}

\newcommand\innerprod[2]{\left\langle #1, #2\right\rangle}
\newcommand\ninnerprod[2]{\langle #1, #2\rangle}

\newcommand\N{\mathbb{N}}
\newcommand\R{\mathbb{R}}

\newcommand\B{\mathbb{B}}

\newcommand\XX{\mathscr X}
\newcommand\YY{\mathscr Y}
\newcommand\PP{\mathscr P}
\newcommand\NN{\mathscr N}
\newcommand\VV{\mathscr V}
\newcommand\WW{\mathscr W}
\newcommand\OON{\mathscr O}

\newcommand\LL{\mathcal L}
\newcommand\KK{\mathcal K}
\newcommand\QQ{\mathcal Q}
\renewcommand\SS{\mathcal S}
\newcommand\HH{\mathcal H}

\newcommand\tto{\rightrightarrows}

\newcommand{\cl}{\operatorname{cl}}

\newcommand{\dist}{\operatorname{dist}}

\newcommand{\cone}{\operatorname{cone}}

\newcommand{\dom}{\operatorname{dom}}

\newcommand{\gph}{\operatorname{gph}}
\newcommand{\epi}{\operatorname{epi}}

\DeclareMathAlphabet{\mathpzc}{OT1}{pzc}{m}{it}
\newcommand\oo{\mathpzc{o}}

\makeatletter
\long\def\@firstoffiveparen#1#2#3#4#5{\textup{\tagform@{#1}}}
\def\eqref@nolink#1{\textup{\tagform@{\ref*{#1}}}}
\def\eqref@link#1{%
\Hy@safe@activestrue
\expandafter\@setref\csname r@#1\endcsname\@firstoffiveparen{#1}%
\Hy@safe@activesfalse
}
\protected\def\eqref{\@ifstar\eqref@nolink\eqref@link}
\makeatother

\title{On error bounds and noncritical Lagrange multipliers}

\author{Mat\'{u}\v{s} Benko%
	\thanks{%
		Johann Radon Institute for Computational and Applied Mathematics,
		4040 Linz,
		Austria,
		\email{matus.benko@ricam.oeaw.ac.at},
		\orcid{0000-0003-3307-7939}%
		}
	\and
	Patrick Mehlitz%
	\thanks{%
		Dresden University of Technology,
		Institute of Numerical Mathematics,
		01062 Dresden,
		Germany,
		\email{patrick.mehlitz@tu-dresden.de},
		\orcid{0000-0002-9355-850X}%
		}
	}
\shortauthor{M.\ Benko \and P.\ Mehlitz}

\acknowledgements{
    The research of Mat\'u\v{s} Benko was supported by the Austrian Science Fund (FWF) 
	under grant P32832-N
	as well as by the infrastructure of the Institute of Computational Mathematics, 
	Johannes Kepler University Linz, Austria.
	\\
	\texttt{Claude} was used during manuscript preparation to assist with language editing 
	and improve clarity, readability, and reasoning.
	All mathematical content, proofs, arguments, references, and final wording were reviewed 
	and verified by the authors, who take full responsibility for the content.
}

\begin{document}

\maketitle

\begin{abstract}
    Critical Lagrange multipliers are known to be responsible for the potentially slow convergence of
	Newton-type methods, and noncriticality of a given multiplier has been tied to
	an error bound which estimates the distance of a primal-dual pair associated with
	a perturbed problem to the primal solution and the multiplier set of the
	unperturbed one, and is decisive for a local convergence analysis.
	We investigate this relation in an abstract setting where merely a set-valued
	mapping assigning primal-dual pairs to a parameter is available, and characterize
	the error bound of interest in terms of noncriticality of the underlying
	multiplier and two calmness-type properties of the associated multiplier mapping,
	one of which is a novel weak inner calmness in the fuzzy sense.
	No structural assumptions on the underlying problem are needed for that,
	and the two calmness-type conditions are not only sufficient but also necessary.
	Afterwards, we specify these findings for composite optimization problems.
	Whenever the subdifferential of the outer function is a polyhedral mapping,
	both conditions hold automatically, which explains the equivalence of
	noncriticality and the error bound observed in the polyhedral setting.
	Whenever the outer function is $C^2$-decomposable, they are secured by calmness
	of a restricted multiplier mapping together with a closedness condition,
	which recovers and extends the known characterization
	for $C^2$-cone reducible constraint systems.
\end{abstract}


\begin{subclass}
	\mscLink{49J52}
	\subclassAnd
	\mscLink{49J53}
	\subclassAnd
	\mscLink{90C30}
\end{subclass}

\section{Introduction}\label{sec:introduction}

In constrained or composite optimization,
the concept of (Lagrange) multipliers is fundamental for the derivation of optimality conditions,
the investigation of duality relations, and the construction of numerical solution algorithms 
like augmented Lagrangian, Lagrange--Newton, or sequential quadratic programming methods.
In this regard, the actual optimization variables, which we will denote by $x\in\R^n$ throughout,
may be interpreted as primal objects while the multipliers, represented by $y\in\R^m$,
play the role of dual objects.
Although one is typically interested in the identification of solution candidates $\bar x\in\R^n$,
optimality conditions and solution methods usually deliver primal-dual pairs $(\bar x,\bar y)$,
so that the multiplier enters the picture as an implicit variable whose presence is
owed to the chosen problem formulation rather than to intrinsic interest in its
value, see \cite{BenkoMehlitz2021,Mehlitz2026} for systematic studies of
implicit variables in mathematical optimization.
In situations where the multiplier $\bar y$ is not uniquely determined
in $\overline Y$, the set of multipliers associated with $\bar x$,
it has been recognized that the speed of convergence
for the numerical identification of $(\bar x,\bar y)$ 
heavily depends on the multiplier $\bar y\in\overline Y$,
which led to coining the notions of critical and noncritical multipliers, 
see the classical paper \cite{Izmailov2005} for the origin of these concepts.
Indeed, it is well documented that Newton-type methods 
are attracted by critical multipliers,
and that this attraction spoils the superlinear convergence one would expect otherwise,
see \cite{Izmailov2005,IzmailovSolodov2009} for the discovery of this phenomenon and
\cite{IzmailovSolodov2015} for a survey.
In standard nonlinear programming,
a multiplier $\bar y$ associated with $\bar x$
is noncritical if and only if a certain primal-dual error bound holds,
see \cite{IzmailovSolodov2012,IzmailovSolodov2015}.
This error bound estimates the distance of a primal-dual pair $(x,y)$
associated with a perturbed problem, parametrized in $p\in\R^\ell$,
to $\{\bar x\}\times\overline Y$, corresponding to the reference perturbation $\bar p\in\R^\ell$,
from above by a multiple of the distance from $p$ to $\bar p$.
Let us point out that the different roles of the primal and the dual variables are
already mirrored by this estimate:
while $x$ is compared with the single point $\bar x$,
the multiplier $y$ is merely compared with the whole set $\overline Y$.
Meanwhile, the notions of critical and noncritical multipliers have been extended 
far beyond standard nonlinear programming,
to variational systems, see \cite{MordukhovichSarabi2018,MordukhovichSarabi2019},
to composite optimization problems with data inducing a certain piecewise linear-quadratic structure,
see \cite{DoMordukhovichSarabi2021,Sarabi2022},
$C^2$-cone reducibility, see \cite{MordukhovichSarabi2019}, 
or $C^2$-decomposability, see \cite{HangSarabi2023},
and to general composite optimization problems, see \cite{BenkoMehlitz2024}.
In each of these settings, the key question has remained the same:
is noncriticality still equivalent to the error bound?
As long as the data is ``polyhedral enough'', which we will make precise below, 
the answer is affirmative and comes without any further assumptions.
Beyond polyhedrality, however, the question becomes a genuine challenge
as the equivalence has been established only for specific problem classes,
only under additional assumptions,
and only by means of long and technical arguments which are tailored to the respective setting.

In the first part of this paper,
we start from a rather abstract setting in order to characterize the aforementioned decisive
error bound in terms of noncriticality. Therefore, we merely consider a set-valued mapping
$M\colon\R^\ell\tto\R^n\times\R^m$ which,
when interpreted in practically relevant scenarios, 
assigns to each parameter $p\in\R^\ell$
the associated primal-dual pairs $(x,y)\in M(p)$.
Decoupling the overall analysis from the precise form of the underlying optimization problem
allows us to get to the bottom of the issue, unburdened by anything like
problem structure, differentiability properties, or validity of constraint qualifications.
A first important result, \cref{thm:full_characterization_of_error_bound}, 
reveals that validity of the error bound is equivalent to noncriticality of the involved multiplier
together with the calmness of the multiplier mapping $Y\colon\R^\ell\times\R^n\tto\R^m$,
given by $Y(p,x)=\{y\,|\,(x,y)\in M(p)\}$,
and a novel weak inner calmness property of $Y$.
Let us emphasize that this characterization does not rely on any structural properties
of $M$, and that the two calmness-type conditions are not only sufficient
in the presence of noncriticality but also necessary.
In this sense, \cref{thm:full_characterization_of_error_bound} uncovers the very
mechanism behind the interplay of noncriticality and error bounds.

The second part of this paper is devoted to the specification of this result
in the framework of composite optimization.
Therefore, we consider the composite problem
\begin{equation}\label{eq:composite_problem}\tag{CP}
	\min\limits_x\quad f_0(x) + g(F(x))\quad\text{s.t.}\quad x\in\R^n,
\end{equation}
where $f_0\colon\R^n\to\R$ and $F\colon\R^n\to\R^m$ are sufficiently smooth while $g\colon\R^m\to\overline{\R}$ is assumed to be proper and lower semicontinuous.
Precise assumptions on the data functions are formulated in \cref{ass:composite_problem_basic} below.
Let us mention that, in the situation where $g$ is convex, this setting of composite optimization 
has been named generalized nonlinear programming in, e.g., 
\cite{BenkoRockafellar2024,Rockafellar2023,Rockafellar2023b} as it covers, exemplary,
standard nonlinear, nonlinear second-order cone, as well as sparse optimization, and many more.
Allowing $g$ to be nonconvex opens the model up to further applications from data science, e.g.,
if $g$ is chosen to be the $\ell_p$-quasi-norm for $p\in[0,1)$.
Numerical algorithms for the computational solution of \eqref{eq:composite_problem},
like augmented Lagrangian, see e.g.\
\cite{BolteSabachTeboulle2018,DeMarchiJiaKanzowMehlitz2023,DeMarchiMehlitz2024,HangSarabi2021,HangSarabi2023,Rockafellar2023,Rockafellar2023b,SabachTeboulle2019}, 
or sequential quadratic programming methods, see e.g.\
\cite{MohammadiMordukhovichSarabi2021,Sarabi2022}, 
focus on the identification of so-called stationary points of \eqref{eq:composite_problem}.
These are points $\bar x\in\R^n$ 
for which a multiplier $\bar y\in\R^m$ exists such that
\begin{equation}\label{eq:stationarity}
	\nabla_x \LL(\bar x,\bar y)=0,\quad \bar y\in\partial g(F(\bar x))
\end{equation}
is valid. Above, $\LL\colon\R^n\times\R^m\to\R$ given by
\begin{equation}\label{eq:Lagrangian}
	\LL(x,y):= f_0(x)+\innerprod{y}{F(x)}
\end{equation}
is referred to as the Lagrangian function associated with \eqref{eq:composite_problem},
and $\partial g(F(\bar x))$ is the limiting subdifferential of $g$ at $F(\bar x)$.
It is well known that whenever $\bar x$ is a local minimizer of \eqref{eq:composite_problem}
satisfying a suitable qualification condition, e.g.,
\begin{equation*}
	F'(\bar x)^\top y=0,\,y\in\partial^\infty g(F(\bar x))
	\quad\Longrightarrow\quad
	y=0,
\end{equation*}
then $\bar x$ is, indeed, stationary.
Here, $\partial^\infty g(F(\bar x))$ represents 
the singular limiting subdifferential of $g$ at $F(\bar x)$.
Hence, in the setting of composite optimization, 
the mapping $M$ from above assigns to parameters $p\in\R^\ell$ the set
of all primal-dual pairs which solve a parametric in $p$ version of \eqref{eq:stationarity}.
Details are provided at the beginning of \cref{sec:composite_optimization}.

We embed our findings in those ones from the literature which address \eqref{eq:composite_problem}.
To the best of our knowledge, the only paper which is concerned with noncritical multipliers
in the full generality of \eqref{eq:composite_problem} is our precursor \cite{BenkoMehlitz2024},
where a similar inner calmness-type assumption already plays a key role.
Its focus, however, is on characterizing whether or not all multipliers associated
with $\bar x$ are noncritical, which complements the standard approach, pursued here,
of tying noncriticality of an individual multiplier to the error bound of interest.
Results of the latter type are available only for particular classes of the function $g$.
In \cite{DoMordukhovichSarabi2021,MordukhovichSarabi2018,Sarabi2022},
the authors consider some special situations which have in common that
the graph of the subdifferential mapping $\partial g\colon\R^m\tto\R^m$ is a polyhedral set.
In these papers, it is proven that validity of the error bound and noncriticality of the
chosen multiplier are actually the same.
We justify these findings in the light of \cref{thm:full_characterization_of_error_bound} 
by proving in \cref{sec:gph_subdiff_polyhedral} that
the additional calmness-type assumptions are inherently satisfied whenever $\partial g$ 
has a polyhedral graph,
which leads to the characterization stated in \cref{thm:gph_of_subdifferential_polyhedral}. 
This also covers earlier contributions which address standard
nonlinear programming, where, in \eqref{eq:composite_problem},
the role of $g$ is played by the indicator function of a special convex polyhedral cone, 
see e.g.\ \cite{Izmailov2005,IzmailovKurennoySolodov2015,IzmailovSolodov2009,IzmailovSolodov2012,IzmailovSolodov2015}.
A second group of results in the literature 
is concerned with the setting where the function $g$ is so-called $C^2$-decomposable,
see \cite{Shapiro2003} for the origin of this concept.
In \cite[Theorem~5.6]{MordukhovichSarabi2019},
where $g$ is the indicator function of a $C^2$-cone reducible set,
validity of the error bound has been shown to be equivalent to the noncriticality of the
chosen multiplier, provided that a restricted multiplier mapping is calm and that
some closedness assumption is satisfied,
while related error bound results for $C^2$-decomposable functions,
obtained under a second-order sufficient condition, can be found in \cite{HangSarabi2023}.
The latter follow a very interesting alternative route which,
relying on the novel concept of a semi-strict graphical derivative,
does not require any closedness assumption.
Again, we are in a position to show in \cref{sec:C2_cone_reducible} that these additional
properties enter the analysis precisely through the calmness-type assumptions postulated
in \cref{thm:full_characterization_of_error_bound},
so that the characterization from \cite[Theorem~5.6]{MordukhovichSarabi2019} is fully
recovered while the auxiliary results
\cite[Lemmas~5.2--5.5]{MordukhovichSarabi2019} become superfluous.
Beyond that, this route remains viable far outside the reducible setting:
in \cref{sec:decomposable}, we obtain a complete characterization of the error bound
for the entire class of $C^2$-decomposable functions,
see \cref{thm:decomposable_characterization}.

This paper, thus, provides a complete and fundamental understanding of the connection between
noncriticality and primal-dual error bounds,
and we show that all characterizations of noncriticality in terms of error bounds
which are available in the literature are consequences of it.

The remainder of the paper is structured as follows.
In \cref{sec:preliminaries}, we summarize the notation used in this paper,
recall some calmness-type notions for set-valued mappings, and introduce a novel notion
of weak inner calmness.
\Cref{sec:main} is dedicated to the analysis in the abstract setting.
The stage is set in \cref{sec:critical_multipliers} where we introduce the error bounds of interest
as well as a rather abstract notion of noncriticality.
The decisive result, which characterizes these error bounds in terms of noncriticality
and additional calmness-type conditions of the multiplier mapping, 
is stated and proven in \cref{sec:main_results},
a discussion follows in \cref{sec:main_results_discussion}.
Furthermore, the results are embedded into the context of parametric optimization
in \cref{sec:perturbations}.
In \cref{sec:composite_optimization}, 
we investigate our earlier findings in the context of composite optimization.
The setting is detailed in \cref{sec:basics_CP},
where we also present some basic observations.
Some particular emphasis is laid on the settings where the function $g$ 
in \eqref{eq:composite_problem} possesses a subdifferential whose graph is a polyhedral set,
see \cref{sec:gph_subdiff_polyhedral}, and where $g$ is $C^2$-decomposable,
see \cref{sec:beyond_polyhedrality}.
\Cref{sec:conclusions} closes the paper with some concluding remarks.

\section{Notation and preliminaries}\label{sec:preliminaries}

In this paper, we mainly rely on standard notation as used in the textbooks 
\cite{DontchevRockafellar2014,Mordukhovich2018,RockafellarWets1998}.

\subsection{Fundamental notation}

Throughout, we equip $\R^n$, the space of all real vectors with $n\in\N$ components,
with the Euclidean inner product $\innerprod{\cdot}{\cdot}$ and the associated Euclidean norm
$\norm{\cdot}$. The closed unit ball of suitable dimensions is denoted by $\B$.
Given $\Omega\subset\R^n$ and $\bar x\in\R^n$,
we exploit $\dist(\bar x,\Omega):=\inf_x\{\nnorm{x-\bar x}\,|\,x\in\Omega\}$ 
to represent the distance of $\bar x$ to $\Omega$,
where $\dist(\bar x,\emptyset):=\infty$ is stipulated.
For brevity of notation,
we exploit $\Omega+\bar x := \bar x + \Omega := \{x+\bar x\,|\,x\in\Omega\}$.
Furthermore, $\cl\Omega$ represents the closure of $\Omega$, and the set
$\cone\Omega:=\{\alpha x\,|\,\alpha\geq 0,\,x\in\Omega\}$ is the cone generated by $\Omega$.
The annihilator of $\bar x$ is given by $\{\bar x\}^\perp := \{y\,|\,\innerprod{y}{\bar x}=0\}$.
The set $\Omega^\circ:=\{y\,|\,\forall x\in\Omega\colon\,\innerprod{y}{x}\leq 0\}$
is the polar cone of $\Omega$ and constitutes a closed, convex cone.
Whenever $\bar x\in\Omega$ holds true,
$T_\Omega(\bar x)$ and $N_\Omega(\bar x)$ denote the tangent cone and
the limiting normal cone to $\Omega$ at $\bar x$, respectively,
and the latter equals the normal cone in the sense of convex analysis if $\Omega$ is a convex set.
Whenever $\Omega$ is a closed, convex cone, $N_\Omega(\bar x)=\Omega^\circ\cap\{\bar x\}^\perp$
is valid.
The transpose and the kernel of a matrix $A\in\R^{m\times n}$ are denoted by 
$A^\top$ and $\ker A:=\{x\,|\,Ax=0\}$, respectively.
Let us recall that a set is said to be polyhedral whenever it can be represented 
as the union of finitely many convex polyhedral sets.
Notation $\oo$ represents a generic function with the property $\lim_{t\downarrow 0}\oo(t)/t = 0$.

In this paper, $\overline \R:=\R\cup\{-\infty,\infty\}$ is the extended real line.
Given an extended-real valued function $\varphi\colon\R^n\to\overline\R$,
$\dom\varphi:=\{x\,|\,\varphi(x)<\infty\}$ and $\epi\varphi:=\{(x,\alpha)\,|\,\varphi(x)\leq\alpha\}$ 
are the domain and the epigraph of $\varphi$, respectively.
We say that $\varphi$ is proper 
if $\dom\varphi\neq\emptyset$ and $\varphi(x)>-\infty$ for all $x\in\R^n$.
Furthermore, $\varphi$ is referred to as lower semicontinuous (lsc) if $\epi\varphi$ is closed.
Given $\bar x\in\dom\varphi$, 
the sets $\partial\varphi(\bar x)$ and $\partial^\infty\varphi(\bar x)$ represent 
the limiting and singular limiting subdifferential of $\varphi$ at $\bar x$, respectively.
For a nonempty, closed set $\Omega\subset\R^n$, $\delta_\Omega\colon\R^n\to\overline\R$ denotes
the indicator function of $\Omega$ which vanishes on $\Omega$ 
and takes value $\infty$ on $\R^n\setminus\Omega$.
We note that $\delta_\Omega$ is proper and lsc, 
and, for some point $\bar x\in\Omega$, it holds
$\partial\delta_\Omega(\bar x)=\partial^\infty\delta_\Omega(\bar x)=N_\Omega(\bar x)$.

Whenever $H\colon\R^n\to\R^m$ is continuously differentiable and $\bar x$ is arbitrarily chosen,
then the matrix $H'(\bar x)\in\R^{m\times n}$ denotes the Jacobian of $H$ at $\bar x$.
For a vector $\xi\in\R^m$, $\innerprod{\xi}{H}\colon\R^n\to\R$ given by
$\innerprod{\xi}{H}(x):=\innerprod{\xi}{H(x)}$ for all $x\in\R^n$ is called
the scalarization mapping of $H$ for $\xi$.
For a real-valued continuously differentiable function $h\colon\R^n\to\R$,
$\nabla h(\bar x):=h'(\bar x)^\top$ is the gradient of $h$ at $\bar x$.
Whenever $h$ is twice continuously differentiable,
we use $\nabla^2 h(\bar x):=(\nabla h)'(\bar x)$ to represent the Hessian of $h$ at $\bar x$.
Note that $\nabla\innerprod{\xi}{H}(\bar x)=H'(\bar x)^\top\xi$, and
$\nabla^2\innerprod{\xi}{H}(\bar x)=\sum_{i=1}^m\xi_i\nabla^2H_i(\bar x)$ holds
if $H$ is twice continuously differentiable,
where $H_1,\ldots,H_m\colon\R^n\to\R$ denote the component functions of $H$.
Finally, let us mention that partial derivatives w.r.t.\ certain blocks of variables
are represented in the canonical way using indices.
	
\subsection{Set-valued mappings}

For a set-valued mapping $\Upsilon\colon\R^n\tto\R^m$,
$\dom\Upsilon:=\{x\,|\,\Upsilon(x)\neq\emptyset\}$ and $\gph\Upsilon:=\{(x,y)\,|\,y\in\Upsilon(x)\}$
denote the domain and the graph of $\Upsilon$, respectively.
Whenever $\gph\Upsilon$ is a polyhedral set, $\Upsilon$ is said to be a polyhedral mapping.
Furthermore, the set-valued mapping $\Upsilon^{-1}\colon\R^m\tto\R^n$ given by 
$\gph\Upsilon^{-1}:=\{(y,x)\,|\,(x,y)\in\gph\Upsilon\}$ is referred to as the inverse of $\Upsilon$.
Let us fix $(\bar x,\bar y)\in\gph\Upsilon$.
Then $\Upsilon$ is called calm at $(\bar x,\bar y)$ whenever there exist neighborhoods
$\XX\subset\R^n$ of $\bar x$ and $\YY\subset\R^m$ of $\bar y$ as well as a constant $\kappa>0$
such that
\[
	\forall x\in\XX,\,\forall y\in \Upsilon(x)\cap\YY\colon\quad
	\dist(y,\Upsilon(\bar x))\leq\kappa\nnorm{x-\bar x}.
\]
Replacing the set $\Upsilon(\bar x)$ in this estimate by the singleton $\{\bar y\}$ 
leads to the stronger estimate
\[
	\forall x\in\XX,\,\forall y\in \Upsilon(x)\cap\YY\colon\quad
	\nnorm{y-\bar y}\leq\kappa\nnorm{x-\bar x}
\]
which is referred to as the isolated calmness of $\Upsilon$ at $(\bar x,\bar y)$,
and it implies that $\bar y$ is locally isolated in $\Upsilon(\bar x)$.
The set-valued mapping $D\Upsilon(\bar x,\bar y)\colon\R^n\tto\R^m$ given by
$\gph D\Upsilon(\bar x,\bar y):=T_{\gph\Upsilon}(\bar x,\bar y)$
is called the graphical derivative of $\Upsilon$ at $(\bar x,\bar y)$.
Whenever $\gph\Upsilon$ is closed locally around $(\bar x,\bar y)$,
then $\Upsilon$ is isolatedly calm at $(\bar x,\bar y)$ if and only if
$D\Upsilon(\bar x,\bar y)(0)=\{0\}$ holds, see \cite[Theorem~4E.1]{DontchevRockafellar2014},
and this tractable condition is well known as the Levy--Rockafellar criterion.

In the following definition, a new calmness type property for set-valued mappings is introduced.

\begin{definition}\label{def:weak_fuzzy_inner_calmness}
	Let $\Upsilon\colon\R^n\tto\R^m$ be a set-valued mapping,
	and fix $\bar x\in\dom\Upsilon$.
	For some direction $d^x\in\R^n$ and $\bar y \in \Upsilon(\bar x)$,
	we say that $\Upsilon$ is \emph{weakly inner calm in the fuzzy sense}
	at $(\bar x,\bar y)$ in direction $d^x$ if,
	whenever there exist sequences $t_k\downarrow 0$, $d_k^x\to d^x$, and $y_k\to\bar y$
	with $\bar x+t_kd_k^x\in\dom\Upsilon$ and $y_k\in\Upsilon(\bar x+t_kd_k^x)$
	for all $k\in\N$,
	then there are, among the sequences with these properties,
	ones which, for some constant $\kappa>0$, additionally satisfy
	$\nnorm{y_k-\bar y}\leq\kappa t_k\nnorm{d_k^x}$ for all $k\in\N$.
\end{definition}

Let us note that the new concept from \cref{def:weak_fuzzy_inner_calmness} relaxes the
one of inner calmness in the fuzzy sense from \cite[Section~4]{BenkoMehlitz2022} in two ways.
First, the original definition is put into a directional context by incorporating direction $d^x$.
Second, it holds automatically if there do not exist sequences 
$t_k\downarrow 0$, $d^x_k\to d^x$, and $y_k\to\bar y$ such that $y_k\in\Upsilon(\bar x + t_kd^x_k)$
is valid for all $k\in\N$.

Obviously, whenever $d^x\notin T_{\dom\Upsilon}(\bar x)$, 
then $\Upsilon$ is weakly inner calm in the fuzzy sense at $(\bar x,y)$
in direction $d^x$ for each $y\in\Upsilon(\bar x)$.
Furthermore, $\Upsilon$ is always weakly inner calm in the fuzzy sense at
$(\bar x,\bar y)$ in direction $d^x=0$,
and whenever $\Upsilon$ is weakly inner calm in the fuzzy sense at $(\bar x,\bar y)$
in direction $d^x\neq 0$, then it is also weakly inner calm in the fuzzy sense
in direction $\alpha d^x$ for each $\alpha>0$.
Hence, one could restrict the directions in \cref{def:weak_fuzzy_inner_calmness}
to the unit sphere without losing any information.

Let us conclude this subsection by relating weak inner calmness in the fuzzy sense
in nonzero directions to the graphical derivative.
The key observation behind the sufficiency part is that tangent directions
come along with realizing sequences, and these exhibit the postulated rate.

\begin{lemma}\label{Lem:weak_fuzzy_ic}
	Let $\Upsilon\colon\R^n\tto\R^m$ be a set-valued mapping,
	fix $(\bar x,\bar y)\in\gph\Upsilon$ and a direction $d^x \in \R^n$ such that $d^x\neq 0$,
	and assume that there exist sequences $t_k\downarrow 0$, $d_k^x\to d^x$, and $y_k\to\bar y$
	with $\bar x+t_kd_k^x\in\dom\Upsilon$ and $y_k\in\Upsilon(\bar x+t_kd_k^x)$
	for all $k\in\N$.
	Then $\Upsilon$ is weakly inner calm in the fuzzy sense
	at $(\bar x,\bar y)$ in direction $d^x$
	if and only if $d^x \in \dom D\Upsilon(\bar x,\bar y)$.
\end{lemma}
\begin{proof}
	Assume first that $\Upsilon$ is weakly inner calm in the fuzzy sense
	at $(\bar x,\bar y)$ in direction $d^x$.
	Then some sequences as postulated additionally satisfy
	$\nnorm{y_k-\bar y}\leq\kappa t_k\nnorm{d_k^x}$
	for all $k\in\N$ and some constant $\kappa>0$.
	Hence, the sequence given by $d^y_k:=(y_k-\bar y)/t_k$, $k\in\N$, is bounded,
	and so it converges to some $d^y\in\R^m$ along a subsequence (without relabeling).
	Due to $(\bar x+t_kd_k^x,\bar y+t_kd^y_k)\in\gph\Upsilon$ for all $k\in\N$,
	this gives $d^y\in D\Upsilon(\bar x,\bar y)(d^x)$, i.e., $d^x\in\dom D\Upsilon(\bar x,\bar y)$.

	Conversely, suppose that $d^x\in\dom D\Upsilon(\bar x,\bar y)$
	and pick $d^y\in D\Upsilon(\bar x,\bar y)(d^x)$.
	By definition of the graphical derivative, there are sequences
	$\tilde t_k\downarrow 0$, $\tilde d_k^x\to d^x$, and $\tilde d^y_k\to d^y$ such that
	$\tilde y_k:=\bar y+\tilde t_k\tilde d^y_k\in\Upsilon(\bar x+\tilde t_k\tilde d_k^x)$ 
	holds for all $k\in\N$.
	Clearly, we have $\tilde y_k\to\bar y$ and, for all sufficiently large $k\in\N$,
	$\nnorm{\tilde d_k^x}\geq\nnorm{d^x}/2$ 
	as well as $\nnorm{\tilde d^y_k}\leq\nnorm{d^y}+1$, so that
	\[
		\nnorm{\tilde y_k-\bar y}
		=
		\tilde t_k\nnorm{\tilde d^y_k}
		\leq
		\frac{2(\nnorm{d^y}+1)}{\nnorm{d^x}}\,\tilde t_k\nnorm{\tilde d_k^x}.
	\]
	Considering, if necessary, the tail of the sequences,
	the requirements of \cref{def:weak_fuzzy_inner_calmness} are met.
\end{proof}

Let us note that, in the absence of sequences as postulated in \cref{Lem:weak_fuzzy_ic},
$\Upsilon$ is trivially weakly inner calm in the fuzzy sense at $(\bar x,\bar y)$
in direction $d^x$, so that, for nonzero directions,
\cref{Lem:weak_fuzzy_ic} provides a complete characterization of this property.

\section{Error bounds versus noncriticality in an abstract setting}\label{sec:main}

In this section, we extend the notions of noncritical and critical multipliers to an arbitrary mapping
$M\colon\R^\ell \tto\R^n\times\R^m$
that assigns a set of pairs $(x,y) \in \R^n\times\R^m$ to a vector of parameters $p \in \R^\ell $.
Throughout the section, 
we assume that $M$ has a closed graph locally around the fixed reference point
$(\bar p, (\bar x, \bar y))\in\gph M$.
Following the reasoning from \cref{sec:introduction},
we perceive $x$ as the true target (the primal solution)
and $y$ as merely the accompanying multiplier (the dual solution),
and this asymmetry is reflected in the error bounds under consideration.
We also state here 
a characterization of these error bounds in terms of noncriticality,
and provide a comprehensive discussion and comparison with the literature.

\subsection{On error bounds and noncriticality}\label{sec:critical_multipliers}

Calmness of $M$ at the reference point corresponds to
the existence of neighborhoods $\PP\subset\R^\ell $ of $\bar p$, $\XX\subset\R^n$ of $\bar x$, and $\YY\subset\R^m$ of $\bar y$,
and a constant $\kappa>0$ such that
\begin{equation*}
	\forall p\in \PP,\,\forall (x,y)\in M(p)\cap (\XX\times\YY)\colon
	\quad \dist((x,y), M(\bar p)) \leq\kappa \norm{p - \bar p}.
\end{equation*}
The isolated calmness of $M$
corresponds to replacing the set $M(\bar p)$ above by the point $(\bar x, \bar y)$, i.e.,
\begin{equation*}
	\forall p\in \PP,\,\forall (x,y)\in M(p)\cap (\XX\times\YY)\colon
	\quad \norm{x-\bar x} + \norm{y-\bar y} \leq\kappa \norm{p - \bar p},
\end{equation*}
and is fully characterized by the Levy--Rockafellar criterion
\begin{equation}\label{eq:isolated_calmness_M1}
	(d^x,d^y) \in DM(\bar p,(\bar x,\bar y))(0)
	\quad\Longrightarrow\quad
	d^x = 0,\,d^y=0.
\end{equation}
The appeal of these conditions lies in the proportional control they provide.
The distance of the primal-dual estimate $(x,y)$ from the primal-dual solution set $M(\bar p)$,
respectively from the pair $(\bar x, \bar y)$, is bounded by the size of perturbations,
and this is essential for convergence analysis of numerical procedures.
At the same time, calmness does not identify any particular primal solution,
while the isolated calmness specifies both as
it necessitates $M(\bar p)\cap (\XX\times\YY) = \{(\bar x,\bar y)\}$.
Practically relevant applications, however, often require the possibility for nonunique multipliers,
yet pinning down the primal solution benefits of the analysis considerably.

We will explore a natural idea of mixing the two modes:
the dependence on $p$ shall be isolatedly calm in $x$ and only calm in $y$.
In order to address the behavior of multipliers, we introduce the multiplier mapping $Y\colon\R^\ell \times\R^n\tto\R^m$ by
\begin{equation*}
	Y(p,x)
	:=
	\{y\,|\,(x,y) \in M(p)\}.
\end{equation*}
Note that $Y$ possesses the same graph as $M$.
Moreover, we denote by
\[
	\overline Y := Y(\bar p,\bar x)
\]
the whole set of multipliers associated with $\bar x$ given the reference perturbation $\bar p$,
which particularly contains $\bar y$.

We are concerned with the following relaxations of isolated calmness,
which no longer carry the undesirable symmetrical treatment of $x$ and $y$.
\begin{definition}\label{def:semi_isolated_calmness_M1}
	\hfill
	\begin{enumerate}
		\item[(i)] We say that $M$ is \emph{partially isolatedly calm w.r.t.\ $x$} at $(\bar p,(\bar x,\bar y))$
			if and only if there exist a neighborhood $\PP\subset\R^\ell $ of $\bar p$,
			neighborhoods $\XX\subset\R^n$ of $\bar x$ and $\YY\subset\R^m$ of $\bar y$,
			and a constant $\kappa>0$
			such that
			\begin{equation}\label{eq:partial_isolated_calmness}
				\forall p\in \PP,\,\forall (x,y)\in M(p)\cap (\XX\times\YY)\colon
				\quad \norm{x-\bar x}\leq\kappa\norm{p-\bar p}.
			\end{equation}
		\item[(ii)] We say that $M$ is \emph{semi-isolatedly calm w.r.t.\ $x$} at $(\bar p,(\bar x,\bar y))$
			if and only if there exist a neighborhood $\PP\subset\R^\ell $ of $\bar p$,
			neighborhoods $\XX\subset\R^n$ of $\bar x$ and $\YY\subset\R^m$ of $\bar y$,
			and a constant $\kappa>0$
			such that
			\begin{equation}\label{eq:semi_isolated_calmness}
				\forall p\in \PP,\,\forall (x,y)\in M(p)\cap (\XX\times\YY)\colon
				\quad\norm{x-\bar x}+\dist(y,\overline Y)\leq\kappa\norm{p-\bar p}.
			\end{equation}
	\end{enumerate}
\end{definition}
The name semi-isolated calmness is justified on account of the notion being stronger than standard calmness
and milder than the isolated calmness. Note also that it controls the distance of $(x,y)$ from the set
\[
	M(\bar p)\cap(\{\bar x\}\times\R^m)
	= \{(x,y)\,|\,y\in Y(\bar p,x)\} \cap(\{\bar x\}\times\R^m)
	= \{\bar x\}\times\overline Y,
\]
which is precisely the type of estimate needed for the local convergence analysis 
in settings with nonunique multipliers.
The notion has been tied to noncriticality of the multiplier since the introduction of the latter,
see \cite{Izmailov2005,IzmailovSolodov2012,IzmailovSolodov2015},
and the connection between them has remained the central theme as noncriticality has been extended
to much broader frameworks, see \cite{DoMordukhovichSarabi2021,HangSarabi2023,MordukhovichSarabi2018,MordukhovichSarabi2019,Sarabi2022}.

We will now clarify the simple relationships between the various calmness properties of $M$,
revealing the role of calmness of $Y$.

\begin{lemma}\label{lem:relations_calmness_x}
\hfill
\begin{enumerate}
		\item[(i)] Mapping $M$ is semi-isolatedly calm w.r.t.\ $x$ at $(\bar p,(\bar x,\bar y))$
			if and only if it is partially isolatedly calm w.r.t.\ $x$ at $(\bar p,(\bar x,\bar y))$
			while $Y$ is calm at $((\bar p,\bar x),\bar y)$.
		\item[(ii)] Mapping $M$ is isolatedly calm at $(\bar p,(\bar x,\bar y))$
			if and only if it is partially isolatedly calm w.r.t.\ $x$ at $(\bar p,(\bar x,\bar y))$
			while $Y$ is isolatedly calm at $((\bar p,\bar x),\bar y)$.
\end{enumerate}
\end{lemma}
\begin{proof}
	We take the obvious implications between various calmness notions for $M$ for granted.
	Comparing these with the calmness notions for $Y$ is also easy, since mappings $M$ and $Y$ share the same graph,
	and so all the calmness notions postulate demands on a localization
	\[
		(p,x,y) \in \gph M \cap (\PP\times\XX\times\YY) = \gph Y \cap (\PP\times\XX\times\YY).
	\]
	
	On the one hand, the trivial estimates
	\[
		\dist(y,\overline Y) \leq \norm{x - \bar x} + \dist(y,\overline Y),
		\quad
		\norm{p - \bar p} \leq \norm{p - \bar p} + \norm{x - \bar x}
	\]
	show that the semi-isolated calmness of $M$ \eqref{eq:semi_isolated_calmness} 
	implies calmness of $Y$, which requires
	\[
		\dist(y,\overline Y) \leq \kappa' \big(\norm{p - \bar p} + \norm{x - \bar x}\big)
	\]
	to hold for some $\kappa'>0$.

	On the other hand, calmness of $Y$, together with the partial isolated calmness of $M$,
	yields the semi-isolated calmness of $M$ with constant $\kappa' + (\kappa' + 1)\kappa$.
	Indeed, add $\norm{x - \bar x}$ to both sides of the calmness estimate for $Y$ and combine it with \eqref{eq:partial_isolated_calmness}
	(and without loss of generality assume that both properties hold in the localization above).

	Moreover, the same arguments clearly work when $\overline Y$ is replaced by $\bar y$,
	and 	this covers the second statement.
\end{proof}

Finally, we want to bring back graphical derivatives and explore how they fit with the new properties. 
The Levy--Rockafellar criterion \eqref{eq:isolated_calmness_M1} for the isolated calmness of $M$
decouples into the conditions
\begin{subequations}\label{eq:decouple_LR}
	\begin{align}
		\label{eq:critical_miltiplier}
		d^y \in DY((\bar p,\bar x),\bar y)(0,d^x)
		\quad&\Longrightarrow\quad
		d^x=0,
		\\
		\label{eq:abstract_uniqueness_condition}
		d^y \in DY((\bar p,\bar x),\bar y)(0,0)
		\quad&\Longrightarrow\quad
		d^y=0
	\end{align}
\end{subequations}
stated in terms of $Y$.
Condition \eqref{eq:abstract_uniqueness_condition} characterizes 
the isolated calmness of $Y$ at $((\bar p,\bar x),\bar y)$,
which particularly means that, in the second statement of \cref{lem:relations_calmness_x},
the partial isolated calmness of $M$ can be replaced by \eqref{eq:critical_miltiplier}.
How to handle such a replacement in the first statement is a first important result of the paper,
stated in \cref{thm:full_characterization_of_error_bound} below.
Condition \eqref{eq:critical_miltiplier} is indeed the key noncriticality criterion,
and we will now reformulate it in terms of the multiplier mapping $Y$.
\begin{definition}\label{def:criticality}
	The multiplier $\bar y\in\overline Y$ is called a \emph{noncritical} multiplier for $\bar x$ 
	if condition
	\[
		(0,d^x) \in \dom D Y((\bar p, \bar x),\bar y)
		\quad\Longrightarrow\quad
		d^x = 0
	\]
	is valid. Otherwise, we refer to $\bar y$ as \emph{critical} for $\bar x$.
\end{definition}

\subsection{Characterizing error bounds in terms of noncriticality}\label{sec:main_results}

In this subsection, we present and prove a result
which interrelates noncriticality of a single multiplier
with the relaxed calmness-type properties of $M$ from \cref{def:semi_isolated_calmness_M1}.

\begin{theorem}\label{thm:full_characterization_of_error_bound}
	Consider the following statements.
	\begin{enumerate}
		\item\label{item:noncritical} 
			The multiplier $\bar y$ is noncritical for $\bar x$.
		\item\label{item:PIC_bar_M} 
			Mapping $M$ is partially isolatedly calm w.r.t.\ $x$ at $(\bar p,(\bar x,\bar y))$.
		\item\label{item:SIC_bar_M} 
			Mapping $M$ is semi-isolatedly calm w.r.t.\ $x$ at $(\bar p,(\bar x,\bar y))$.
	\end{enumerate}
	Then we always have the implications 
	$[\ref{item:SIC_bar_M}\Longrightarrow\ref{item:PIC_bar_M}\Longrightarrow\ref{item:noncritical}]$.
	
	Consider the following calmness-type assumptions. 
	\begin{enumerate}[label=(\roman*)]
		\item\label{item:weak_inner_calmness}
			Mapping $Y$ is weakly inner calm in the fuzzy sense at $((\bar p,\bar x),\bar y)$
			in each direction $(0,d^x)$ such that $d^x\neq 0$,
		\item\label{item:calmness}
			Mapping $Y$ is calm at $((\bar p,\bar x),\bar y)$.
	\end{enumerate}
	Then we have the equivalences 
	$[\ref{item:noncritical}\land\ref{item:weak_inner_calmness} \Longleftrightarrow \ref{item:PIC_bar_M}]$
	and 
	$[\ref{item:PIC_bar_M}\land\ref{item:calmness}\Longleftrightarrow\ref{item:SIC_bar_M}]$.
	Particularly, we have the equivalence 
	$[\ref{item:noncritical}\land\ref{item:weak_inner_calmness}\land\ref{item:calmness}\Longleftrightarrow\ref{item:SIC_bar_M}]$.
\end{theorem}
\begin{proof}
	The equivalence $[\ref{item:PIC_bar_M}\land\ref{item:calmness}\Longleftrightarrow\ref{item:SIC_bar_M}]$
	has been addressed in the first statement of \cref{lem:relations_calmness_x}.
	It remains to show $[\ref{item:noncritical}\land\ref{item:weak_inner_calmness} \Longleftrightarrow \ref{item:PIC_bar_M}]$,
	and we do it by proving three implications below.
	\\
	$[\ref{item:PIC_bar_M}\Longrightarrow\ref{item:noncritical}]$: 
	Pick $(0,d^x)\in \dom DY((\bar p,\bar x),\bar y)$.
	This implies the existence of $d^y\in\R^m$ such that $d^y\in DY((\bar p,\bar x),\bar y)(0,d^x)$.
	By definition of the graphical derivative,
	we find sequences $t_k\downarrow 0$, $p_k\to\bar p$, $x_k\to\bar x$, and $y_k\to\bar y$
	such that $((p_k,x_k),y_k)\in\gph Y$ holds for all $k\in\N$ together with the convergences
	\[
		\frac{p_k-\bar p}{t_k}\to 0,\quad
		\frac{x_k-\bar x}{t_k}\to d^x,\quad
		\frac{y_k-\bar y}{t_k}\to d^y.
	\]
	Note that $(p_k,(x_k,y_k))\in\gph M$ is true for each $k\in\N$.
	According to \ref{item:PIC_bar_M}, there is a constant $\kappa>0$ such that
	\[
		\norm{x_k-\bar x}\leq\kappa\norm{p_k-\bar p}
	\]
	is valid for all large enough $k\in\N$.
	Dividing this by $t_k$ and taking the limit $k\to\infty$ then gives $d^x=0$.
	Hence, \ref{item:noncritical} is valid.
	\\
	$[\ref{item:PIC_bar_M}\Longrightarrow\ref{item:weak_inner_calmness}]$:
	Let us fix $d^x\in\R^n$ such that $d^x\neq 0$ and suppose that there are
	sequences $t_k\downarrow 0$, $d^p_k\to 0$, $d^x_k\to d^x$,
	and $y_k\to\bar y$ such that, for each $k\in\N$, 
	the inclusion $y_k\in Y((\bar p,\bar x)+t_k(d^p_k,d^x_k))$ is valid.
	This particularly means that $(\bar p+t_kd^p_k,(\bar x+t_kd^x_k,y_k))\in\gph M$
	holds for each $k\in\N$.
	According to \ref{item:PIC_bar_M}, there is a constant $\kappa>0$ such that
	\[
		t_k\norm{d^x_k}\leq \kappa t_k\nnorm{d^p_k}
	\]
	is valid for large enough $k\in\N$.
	Dividing this by $t_k$ and taking the limit $k\to\infty$ then gives $d^x=0$,
	a contradiction. Hence, sequences as above do not exist,
	and so $Y$ is trivially weakly inner calm in the fuzzy sense at $((\bar p,\bar x),\bar y)$
	in each direction $(0,d^x)$ such that $d^x\neq 0$,
	i.e., \ref{item:weak_inner_calmness} holds.
	\\
	$[\ref{item:noncritical}\land\ref{item:weak_inner_calmness}\Longrightarrow\ref{item:PIC_bar_M}]$:
	Suppose that \ref{item:PIC_bar_M} fails to hold.
	Thus, there exist sequences $p_k\to \bar p$, $x_k\to\bar x$, and $y_k\to\bar y$
	such that the conditions
	\begin{equation*}
		(x_k,y_k)\in M(p_k),\quad
		x_k\neq \bar x,\quad
		\frac{\norm{p_k-\bar p}}{\norm{x_k - \bar x}}
		<
		\frac{1}{k}
	\end{equation*}
	hold for each $k \in \N$.
	Observe that this, particularly, implies $y_k\in Y(p_k,x_k)$ for each $k\in\N$.
	Setting $t_k := \norm{x_k - \bar x}$ for each
	$k\in\N$, we obtain, along a subsequence (without relabeling), the convergences
	\[
		d_k^p:=\frac{p_k-\bar p}{t_k}\to 0,\quad
		d_k^x:=\frac{x_k-\bar x}{t_k}\to d^x
	\]
	for some $d^x \in \R^n$ with $d^x\neq 0$
	as well as
	\[
		\dist(\bar y, Y((\bar p,\bar x) + t_k (d_k^p,d_k^x))) 
		=
		\dist(\bar y,Y(p_k,x_k))\leq \norm{y_k-\bar y} \to 0.
	\]
	Due to validity of \ref{item:weak_inner_calmness},
	\cref{Lem:weak_fuzzy_ic} yields $(0,d^x) \in \dom DY((\bar p,\bar x),\bar y)$,
	which means that $\bar y$ is critical for $\bar x$.
	That, however, is a contradiction to \ref{item:noncritical}.
\end{proof}

Let us emphasize that the simplicity of the above proof is not a shortcoming.
Once weak inner calmness in the fuzzy sense has been identified as the missing ingredient,
the connection between primal-dual error bounds and noncriticality becomes transparent,
and the actual content of \cref{thm:full_characterization_of_error_bound} lies in the
necessity of the conditions \ref{item:weak_inner_calmness} and \ref{item:calmness}:
these two properties are exactly what has to be secured, in one form or another,
by any result which ties noncriticality to error bounds.
Some first insights into their nature are presented in \cref{sec:main_results_discussion},
and \cref{sec:composite_optimization} demonstrates that they are verifiable in the
relevant settings of composite optimization,
which recovers and extends the characterizations available in the literature.

\subsection{Discussion}\label{sec:main_results_discussion}

Let us put the calmness-type conditions \ref{item:weak_inner_calmness} and \ref{item:calmness}
from \cref{thm:full_characterization_of_error_bound} into perspective.
In several important settings, characterizations of noncriticality via
semi-isolated calmness are available which do not involve any additional assumptions.
Exemplary, this is the case for standard nonlinear programs with equality constraints,
see \cite[Proposition~1]{IzmailovSolodov2012}, and, more generally,
for composite problems of type \eqref{eq:composite_problem} 
where $g$ is convex piecewise linear-quadratic,
see \cite[Theorem~3.6]{Sarabi2022} as well as
\cite[Theorem~5.1]{DoMordukhovichSarabi2021} and \cite[Theorem~4.1]{MordukhovichSarabi2018}
for particular subclasses.
At the first glance, the presence of the conditions \ref{item:weak_inner_calmness}
and \ref{item:calmness} in \cref{thm:full_characterization_of_error_bound}
may, thus, appear to be a shortcoming of the abstract approach.
This is, however, not the case.
As we will demonstrate in \cref{sec:gph_subdiff_polyhedral},
see \cref{prop:gph_of_subdifferential_polyhedral} in particular,
in the setting of composite optimization where $\gph(\partial g)$
is polyhedral, which covers all the results mentioned above,
both conditions are automatically satisfied,
so \cref{thm:full_characterization_of_error_bound} indeed characterizes
semi-isolated calmness in terms of noncriticality alone.
Even beyond polyhedrality, both conditions can be secured under mild
additional assumptions, as we will demonstrate in \cref{sec:beyond_polyhedrality}
for composite problems governed by $C^2$-decomposable functions.
While condition \ref{item:weak_inner_calmness} will be verified directly
in \cref{prop:weak_inner_calmness_decomposable},
the mechanism behind condition \ref{item:calmness} is more subtle,
and we start by uncovering it here.

Interestingly, it is typically not the mapping $Y$ itself
but its restriction to the fixed primal reference point $\bar x$,
i.e., the restricted multiplier mapping $Y_{\bar x}\colon\R^\ell \tto\R^m$ given by
\[
	Y_{\bar x}(p)
	:=
	Y(p,\bar x),
\]
whose calmness is usually assumed in the literature.
Note that $Y_{\bar x}(\bar p)=\overline Y$.
Exemplary, in \cite[Theorem~5.6]{MordukhovichSarabi2019},
semi-isolated calmness is derived from noncriticality of $\bar y$ for $\bar x$,
calmness of $Y_{\bar x}$, and an additional closedness condition
for composite problems with so-called $C^2$-cone reducible structure.
Calmness of $Y_{\bar x}$ is also among the assumptions
of the error bound results in \cite[Theorem~2]{HangSarabi2023},
which come without any closedness condition at all;
we treat these settings in detail in \cref{sec:beyond_polyhedrality}.
The restricted multiplier mapping also arises naturally from the viewpoint of polyhedrality:
fixing $x:=\bar x$ freezes, in the composite setting, 
all nonlinearity caused by the inner smooth function
in the system which describes $\gph Y$,
and whenever $\gph(\partial g)$ is polyhedral,
$\gph Y_{\bar x}$ is polyhedral as well.
Calmness of $Y_{\bar x}$ would, thus, be automatic by Robinson's classical result
on polyhedral mappings, see \cite[Proposition~1]{Robinson1981}
and \cref{rem:calmness_of_Yx} for details.
The graph of $Y$, in contrast, depends nonlinearly on $x$,
and Robinson's result is not applicable to $Y$ itself.
This raises the question of whether calmness of $Y_{\bar x}$ can take over
the role of condition \ref{item:calmness} in \cref{thm:full_characterization_of_error_bound}.
Let us mention that calmness of $Y$ at $((\bar p,\bar x),\bar y)$ always implies
calmness of $Y_{\bar x}$ at $(\bar p,\bar y)$,
which readily follows by fixing $x:=\bar x$ in the associated estimates.
The following example shows that the converse implication fails in general and that,
moreover, condition \ref{item:calmness} cannot be naively replaced by
calmness of $Y_{\bar x}$ in the general situation considered here.

\begin{example}\label{ex:calmness_Y_restricted_not_enough}
	Let us consider $\ell:=n:=m:=1$,
	the mapping $M\colon\R\tto\R\times\R$ given by
	\[
		\gph M
		:=
		\bigl(\R\times\{(0,0)\}\bigr)
		\cup
		\{(t,(t,\sqrt t))\,|\,t\geq 0\}
	\]
	whose graph is closed,
	and the reference point $(\bar p,(\bar x,\bar y)):=(0,(0,0))\in\gph M$.
	By construction, we have $Y_{\bar x}(p)=\{0\}$ for each $p\in\R$,
	so $Y_{\bar x}$ is calm at $(\bar p,\bar y)$ while $\overline Y=\{0\}$.
	Furthermore, each pair $(x,y)\in M(p)$ satisfies either $x=0$ or $x=p$,
	so $M$ is partially isolatedly calm w.r.t.\ $x$ at $(\bar p,(\bar x,\bar y))$,
	and $\bar y$ is noncritical for $\bar x$ due to the implication
	$[\ref{item:PIC_bar_M}\Longrightarrow\ref{item:noncritical}]$
	from \cref{thm:full_characterization_of_error_bound}.
	However, for each $t>0$, we have $(t,\sqrt t)\in M(t)$ and, for each fixed $\kappa>0$,
	\[
		\dist(\sqrt t,\overline Y)
		=
		\sqrt t
		>
		\kappa\,t
		=
		\kappa\norm{p-\bar p}
	\]
	for all sufficiently small $t>0$,
	so $M$ is not semi-isolatedly calm w.r.t.\ $x$ at $(\bar p,(\bar x,\bar y))$.
	Hence, in \cref{thm:full_characterization_of_error_bound},
	condition \ref{item:calmness} cannot be replaced by calmness of $Y_{\bar x}$.
	Let us note that, in accordance with \cref{thm:full_characterization_of_error_bound},
	$Y$ indeed fails to be calm at $((\bar p,\bar x),\bar y)$
	as $\dist(\sqrt t,\overline Y)=\sqrt t$ eventually exceeds
	$\kappa(\norm{p-\bar p}+\norm{x-\bar x})=2\kappa\,t$ as well.
\end{example}

The reason behind the failure in \cref{ex:calmness_Y_restricted_not_enough} is easily identified:
the multiplier $\sqrt t$, associated with $x:=t$ and the parameter $p:=t$,
is not a multiplier associated with $\bar x$ for any parameter whatsoever,
so the desirable behavior of $Y_{\bar x}$ simply carries no information
about the points of interest.
Whether or not this phenomenon can occur clearly depends on the parametrization,
i.e., on the way the parameter $p$ enters the mapping $M$.
In the composite setting, the parameters comprise so-called canonical perturbations
which are rich enough to compensate displacements of the primal variable
while keeping the multiplier fixed,
and this is precisely the feature which we now distill into abstract form.
For these purposes, let us associate with $M$ the parameter mapping
$Q_y\colon\R^n\tto\R^\ell $ given, for each $y\in\R^m$, by
\[
	Q_y(x)
	:=
	\{p\,|\,(x,y)\in M(p)\},
\]
so $y\in Y(p,x)$ holds if and only if $p\in Q_y(x)$.

\begin{definition}\label{def:parametric_calmness}
	We say that $M$ is \emph{parametrically calm w.r.t.\ $x$} at $(\bar p,(\bar x,\bar y))$
	if and only if there exist a neighborhood $\PP\subset\R^\ell $ of $\bar p$,
	neighborhoods $\XX\subset\R^n$ of $\bar x$ and $\YY\subset\R^m$ of $\bar y$,
	and a constant $\kappa>0$ such that
	\begin{equation}\label{eq:parametric_calmness}
		\forall x\in\XX,\,\forall y\in\YY\colon\quad
		Q_y(x)\cap\PP
		\subset
		Q_y(\bar x)+\kappa\norm{x-\bar x}\B.
	\end{equation}
\end{definition}

In other words, \eqref{eq:parametric_calmness} postulates that whenever $y$ is a multiplier
associated with some $x$ close to $\bar x$ for a parameter $p$ close to $\bar p$,
then $y$ is a multiplier associated with $\bar x$ itself for some parameter $p'$
which satisfies $\norm{p'-p}\leq\kappa\norm{x-\bar x}$;
the terminology reflects that this requires nothing else but calmness of the
mappings $Q_y$ at $\bar x$, holding uniformly w.r.t.\ $y\in\YY$.
Parametric calmness is a condition on the parametrization only.
In \cref{ex:calmness_Y_restricted_not_enough}, it fails since, for each $t>0$,
we have $Q_{\sqrt t}(t)=\{t\}$ while $Q_{\sqrt t}(\bar x)=\emptyset$.
The following remark provides a sufficient condition which underlines
that parametric calmness concerns the richness of the parametrization.

\begin{remark}\label{rem:ample_parametrization}
	Assume that, locally around $(\bar p,(\bar x,\bar y))$, $\gph Y$ admits a representation
	\begin{equation}\label{eq:ample_representation}
		\gph Y
		=
		\{((p,x),y)\,|\,(S(p,x,y),T(y))\in\Omega\}
	\end{equation}
	for some set $\Omega\subset\R^{s}\times\R^{t}$,
	an arbitrary mapping $T\colon\R^m\to\R^{t}$,
	and a continuously differentiable mapping
	$S\colon\R^\ell \times\R^n\times\R^m\to\R^{s}$
	such that $S'_p(\bar p,\bar x,\bar y)$ possesses full row rank $s$.
	Then $M$ is parametrically calm w.r.t.\ $x$ at $(\bar p,(\bar x,\bar y))$.
	Indeed, by the parametric Lyusternik--Graves theorem
	\cite[Theorem~5G.1]{DontchevRockafellar2014},
	the mappings $S(\cdot,\bar x,y)$ are metrically regular 
	at $(\bar p,S(\bar p,\bar x,y))$ with a constant and neighborhoods
	which do not depend on $y$ close to $\bar y$,
	so that, for each $((p,x),y)\in\gph Y$ sufficiently close to $((\bar p,\bar x),\bar y)$,
	the equation $S(p',\bar x,y)=S(p,x,y)$ possesses a solution $p'$ with
	$\norm{p'-p}\leq\kappa\norm{x-\bar x}$, where $\kappa>0$ incorporates
	a local Lipschitz constant of $S$ w.r.t.\ $(p,x)$;
	as the component $T(y)$ is not affected, this gives $p'\in Q_y(\bar x)$,
	and \eqref{eq:parametric_calmness} follows.
	Observe that it is essential that the second component mapping $T$
	in \eqref{eq:ample_representation} does not depend on $x$:
	portions of the system describing $\gph Y$ which vary in $x$
	but cannot be controlled through the parameter $p$
	are precisely what parametric calmness rules out,
	see \cref{ex:calmness_Y_restricted_not_enough} again.
	The postulated full row rank assumption on $S'_p(\bar p,\bar x,\bar y)$ mirrors
	the concept of ample parametrization known from the theory of generalized equations,
	see \cite[Section~2C]{DontchevRockafellar2014}.
\end{remark}

We are now in a position to clarify the relation between the calmness properties
of $Y$ and $Y_{\bar x}$.

\begin{lemma}\label{lem:parametric_calmness_transfer}
	Let $M$ be parametrically calm w.r.t.\ $x$ at $(\bar p,(\bar x,\bar y))$,
	and let $Y_{\bar x}$ be calm at $(\bar p,\bar y)$.
	Then $Y$ is calm at $((\bar p,\bar x),\bar y)$.
\end{lemma}
\begin{proof}
	Due to $Y_{\bar x}(\bar p)=\overline Y$, the assumed calmness of $Y_{\bar x}$
	yields a constant $\kappa'>0$ such that the inequality
	$\dist(y,\overline Y)\leq\kappa'\norm{p-\bar p}$ holds for all
	$p$ close to $\bar p$ and all $y\in Y_{\bar x}(p)$ close to $\bar y$.
	Now, pick $p$, $x$, and $y\in Y(p,x)$ sufficiently close to
	$\bar p$, $\bar x$, and $\bar y$, respectively.
	Due to $p\in Q_y(x)$, the imposed parametric calmness yields some
	$p'\in Q_y(\bar x)$, i.e., $y\in Y_{\bar x}(p')$,
	such that $\nnorm{p'-p}\leq\kappa\norm{x-\bar x}$,
	and $p'$ is close to $\bar p$ as
	$\nnorm{p'-\bar p}\leq\norm{p-\bar p}+\kappa\norm{x-\bar x}$.
	Consequently, we obtain
	\[
		\dist(y,\overline Y)
		\leq
		\kappa'\nnorm{p'-\bar p}
		\leq
		\kappa'\max(1,\kappa)\bigl(\norm{p-\bar p}+\norm{x-\bar x}\bigr),
	\]
	and, due to $Y(\bar p,\bar x)=\overline Y$,
	$Y$ is calm at $((\bar p,\bar x),\bar y)$.
\end{proof}

Since calmness of $Y$ always implies calmness of $Y_{\bar x}$ as mentioned above,
\cref{lem:parametric_calmness_transfer} particularly shows that both properties
are equivalent whenever $M$ is parametrically calm w.r.t.\ $x$
at $(\bar p,(\bar x,\bar y))$.
This is precisely the mechanism announced above:
in the composite setting, parametric calmness holds
without any assumptions on the function $g$,
see \cref{prop:parametric_calmness_composite} for the direct verification
with explicit constants and \cref{cor:calmness_Y_equivalence_composite}
for the resulting equivalence.
Together with Robinson's result mentioned earlier, this yields that
condition \ref{item:calmness} of \cref{thm:full_characterization_of_error_bound}
is inherently satisfied in the composite setting whenever
$\gph(\partial g)$ is polyhedral, see \cref{prop:calmness_Y_composite_polyhedral}.

Let us now address condition \ref{item:weak_inner_calmness}
of \cref{thm:full_characterization_of_error_bound} in more detail.
An interesting feature of noncriticality as introduced in \cref{def:criticality}
is that it concerns only the single multiplier $\bar y$,
being imposed via the graphical derivative of $Y$ at $((\bar p,\bar x),\bar y)$.
Since the graphical derivative is not robust w.r.t.\ perturbations of the reference point,
on the one hand,
noncriticality of $\bar y$ alone does not carry information about the remaining
multipliers in $\overline Y$, and it may indeed happen that a noncritical multiplier
is the limit of critical ones, see \cref{ex:noncritical_limit_of_critical} below.
The partial isolated calmness of $M$ w.r.t.\ $x$, on the other hand,
is robust in this respect: the estimate \eqref{eq:partial_isolated_calmness}
does not depend on the reference multiplier $\bar y$,
which enters the definition only through the neighborhood $\YY$.
Consequently, if $M$ is partially isolatedly calm w.r.t.\ $x$ at $(\bar p,(\bar x,\bar y))$,
then it enjoys this property at $(\bar p,(\bar x,y))$ for every $y\in\overline Y\cap\YY$ as well,
and the implication $[\ref{item:PIC_bar_M}\Longrightarrow\ref{item:noncritical}]$
of \cref{thm:full_characterization_of_error_bound} shows that all multipliers
from $\overline Y\cap\YY$ are then noncritical for $\bar x$.
In view of the equivalence
$[\ref{item:noncritical}\land\ref{item:weak_inner_calmness}\Longleftrightarrow\ref{item:PIC_bar_M}]$,
it is thus precisely condition \ref{item:weak_inner_calmness} which upgrades
noncriticality of the single multiplier $\bar y$ to noncriticality of all
multipliers close by, i.e., it makes noncriticality a robust property.

\begin{example}\label{ex:noncritical_limit_of_critical}
	Let us consider $\ell:=n:=m:=1$,
	the set $\overline Y:=\{0\}\cup\{1/k\,|\,k\in\N\}$,
	the mapping $M\colon\R\tto\R\times\R$ given by
	\[
		\gph M
		:=
		\{(0,(0,y))\,|\,y\in\overline Y\}
		\cup
		\{(0,(t,1/k+t))\,|\,k\in\N,\,t\in[0,1/k^2]\}
	\]
	whose graph is closed,
	and the reference point $(\bar p,(\bar x,\bar y)):=(0,(0,0))\in\gph M$.
	By construction, $\overline Y$ is indeed the set of multipliers associated
	with $\bar x$, i.e., $\overline Y=Y(\bar p,\bar x)$.
	For each fixed $k\in\N$, the curve $t\mapsto(0,(t,1/k+t))$ generates the direction
	$(0,1)\in\dom DY((\bar p,\bar x),1/k)$,
	so the multiplier $1/k$ is critical for $\bar x$.
	The multiplier $\bar y=0$ is, however, noncritical for $\bar x$.
	Indeed, let $(0,d^x)\in\dom DY((\bar p,\bar x),\bar y)$ be generated by sequences
	$s_j\downarrow 0$ and $((p_j,x_j),y_j)\in\gph Y$ with
	$((p_j,x_j),y_j)\to((\bar p,\bar x),\bar y)$, $x_j/s_j\to d^x$, and $y_j/s_j\to d^y$
	for some $d^y\in\R$, and suppose that $d^x\neq 0$.
	Since $x_j\geq 0$ holds for all $j\in\N$, we get $d^x>0$,
	so $x_j>0$ is valid for all sufficiently large $j\in\N$.
	For such $j$, we find $k_j\in\N$ with $y_j=1/k_j+x_j$ and $x_j\leq 1/k_j^2$,
	where $y_j\to 0$ enforces $k_j\to\infty$,
	and we arrive at
	\[
		\frac{y_j}{s_j}
		\geq
		\frac{1/k_j}{x_j}\cdot\frac{x_j}{s_j}
		\geq
		k_j\,\frac{x_j}{s_j}
		\to
		\infty,
	\]
	contradicting $y_j/s_j\to d^y$.
	Hence, the noncritical multiplier $\bar y$ is the limit of the critical multipliers $1/k$.
	Let us mention that, in accordance with the discussion above,
	$M$ is not partially isolatedly calm w.r.t.\ $x$ at $(\bar p,(\bar x,\bar y))$,
	as the points $(t,1/k+t)\in M(\bar p)$ with $t>0$
	violate \eqref{eq:partial_isolated_calmness},
	and, in view of \cref{thm:full_characterization_of_error_bound},
	condition \ref{item:weak_inner_calmness} must fail.
	In particular, noncriticality of $\bar y$ alone does not imply
	the partial isolated calmness of $M$ w.r.t.\ $x$,
	i.e., condition \ref{item:weak_inner_calmness} cannot be omitted
	in \cref{thm:full_characterization_of_error_bound}.
\end{example}

This also sheds some light on the results from the literature mentioned
at the beginning of this subsection: whenever condition \ref{item:weak_inner_calmness}
is satisfied, as it happens, e.g., if 
the composite setting is considered where $\gph(\partial g)$ is polyhedral or, 
under mild additional assumptions, $g$ is $C^2$-decomposable, 
see \cref{prop:gph_of_subdifferential_polyhedral,prop:weak_inner_calmness_decomposable} below,
such robustness of noncriticality is obtained for free.

\subsection{An exemplary setting from parametric optimization}\label{sec:perturbations}

As mentioned above, our prime motivation stems from the setting where
\begin{equation*}
	M(v,u)
	=
	\{(x,y)\,|\,(v,y) \in \partial \varphi(x,u)\},
	\quad
	Y(v,u,x)
	=
	\{y\,|\,(v,y) \in \partial \varphi(x,u)\}
\end{equation*}
hold
for a proper, lsc function $\varphi\colon \R^n\times\R^m \to \overline{\R}$ 
such that $\partial\varphi\colon\R^n\times\R^m\tto\R^n\times\R^m$ has a closed graph.
Thus, in the language of \cref{sec:main}, the parameter is $p:=(v,u)\in\R^\ell$ with $\ell:=n+m$,
and the reference parameter is $\bar p:=(0,0)$.

Noncriticality of a multiplier $\bar y\in\overline Y$ in the sense of \cref{def:criticality}
can be expressed as
	\begin{equation}\label{eq:noncriticality_param_opt}
		(0,d^y) \in D(\partial \varphi)((\bar x,0),(0,\bar y))(d^x,0)
		\quad\Longrightarrow\quad
		d^x = 0.
	\end{equation}
As we will detail later on in \cref{sec:composite_optimization}, 
in the setting of composite optimization,
this definition coincides with the notion of noncritical multipliers 
from \cite[Definition~4.1]{BenkoMehlitz2024},
and, thus, indeed constitutes an extension of the notion.

This setting not only justifies the terminology, but also provides additional interpretations
in terms of the abstract optimization problem
\begin{equation}\label{eq:abstract_problem_pert}\tag{P$(v,u)$}
	\min\limits_x\quad \varphi(x,u) - \innerprod{v}{x} \quad\text{s.t.}\quad x\in\R^n,
\end{equation}
involving parameters $(v,u)\in\R^n\times\R^m$.
We perceive \eqref{eq:abstract_problem_pert} as a perturbation of the target problem
\hyperref[eq:abstract_problem_pert]{(P$(0,0)$)} of minimizing $\varphi(\cdot,0)$.

Clearly, both $M$ and $Y$ encode stationarity conditions for \eqref{eq:abstract_problem_pert}.
Let us briefly comment on the relation between the stationarity condition
$(v,y) \in \partial \varphi(x,u)$, which involves the multiplier $y$,
and the multiplier-free condition $v \in \partial_x \varphi(x,u)$
featuring the partial limiting subdifferential of $\varphi$ w.r.t.\ $x$.
If $\varphi$ satisfies a suitable constraint qualification at $(x,u)$,
exemplary the condition
\[
	(0,y)\in\partial^\infty\varphi(x,u)
	\quad\Longrightarrow\quad
	y=0,
\]
then $v\in\partial_x\varphi(x,u)$ implies the existence of a multiplier vector
$y \in \R^m$ with $(v,y) \in \partial \varphi(x,u)$, i.e., $(x,y)\in M(v,u)$,
see \cite[Corollary~10.11]{RockafellarWets1998}.
The reverse implication is not automatic,
but it holds, exemplary, in the presence of so-called prox-regularity of $\varphi$
at $(x,u)$ for $(v,y)$, a common assumption in this context,
inherent in the composite setting studied in \cref{sec:composite_optimization} 
whenever the function $g$ is convex,
since then $(v,y)$ is a proximal subgradient of $\varphi$ at $(x,u)$
and the defining quadratic minorization persists when restricted to the slice $u':=u$.
Since the condition $v \in \partial_x \varphi(x,u)$ is necessary
for local optimality of $x$ in \eqref{eq:abstract_problem_pert},
these observations tie the mappings $M$ and $Y$ to minimization.
As an additional step, one can even go back from stationarity to local optimality
by invoking variational sufficiency, a concept to which we turn now.

The aforementioned pattern, stationarity plus a suitable second-order ingredient
yielding local optimality with stability, has recently been formalized by Rockafellar
through the concept of variational sufficiency, see \cite{Rockafellar2019,Rockafellar2023}.
Without going into detail, the (strong) variational sufficient condition for local optimality
in \eqref{eq:abstract_problem_pert} postulates that a quadratically tilted version of $\varphi$
is variationally (strongly) convex at the reference point, meaning that,
in a local primal-dual sense, the subgradients and function values of $\varphi$
interact as if $\varphi$ were convex.
It guarantees local optimality of stationary points without any constraint qualification,
supports augmented Lagrangian methods in nonconvex optimization, see \cite{Rockafellar2023b},
and, in standard nonlinear programming, it is equivalent to the standard strong second-order
sufficient condition, see \cite{Rockafellar2023}.
Most relevant for the above discussion, in the presence of this condition,
the mapping $M$, localized around the reference point,
comprises only pairs whose primal component is a genuine local minimizer
of \eqref{eq:abstract_problem_pert},
see \cite[Theorem~1.1]{BenkoRockafellar2026},
so that our results translate into statements about true solution mappings.
Let us also mention that the discussion in this subsection tacitly concerns
prox-regular functions:
exemplary, the strong variational sufficient condition subsumes
variational strong convexity of a quadratically tilted version of $\varphi$
and, hence, prox-regularity of $\varphi$ at the reference point,
see \cite[Theorem~1.1]{Rockafellar2025}.

Most importantly for us, strong variational sufficiency admits pointwise second-order
criteria of positive-definiteness type. Given a suitable generalized derivative
operator $\mathcal D$, these take the form
\begin{equation}\label{eq:SOSC_robust}
	(d^v,d^y) \in \mathcal D(\partial \varphi)((\bar x,0),(0,\bar y))(d^x,0), \ d^x \neq 0
	\quad\Longrightarrow\quad
	\innerprod{d^v }{d^x} > 0,
\end{equation}
where $\mathcal D$ can be chosen as the strict graphical derivative
or, equivalently, in terms of generalized Hessian bundles,
see \cite[Theorem~5.1]{Rockafellar2025},
while the limiting coderivative and the limiting graphical derivative work as well,
even though they are not employed in \cite{Rockafellar2025}.
A distinctive common feature of these derivative constructions is their robustness
w.r.t.\ perturbations of the reference point,
which the (ordinary) graphical derivative lacks.
In particular, since the graphical derivative is always contained in its strict counterpart,
the criteria \eqref{eq:SOSC_robust} are stronger than the corresponding condition
\begin{equation}\label{eq:SOSC_graph_der}
	(d^v,d^y) \in D(\partial \varphi)((\bar x,0),(0,\bar y))(d^x,0), \ d^x \neq 0
	\quad\Longrightarrow\quad
	\innerprod{d^v }{d^x} > 0,
\end{equation}
which can be equivalently expressed via second subderivatives and has been employed
as a second-order sufficient condition in composite optimization by Sarabi and coauthors,
see e.g.\ \cite{HangSarabi2023,MohammadiMordukhovichSarabi2022}.
Comparing \eqref{eq:SOSC_graph_der} with the 
noncriticality condition \eqref{eq:noncriticality_param_opt}
immediately reveals that \eqref{eq:SOSC_graph_der} implies noncriticality of $\bar y$ for $\bar x$.
The premise of the former with $d^v:=0$ would lead to the absurd conclusion
$0=\innerprod{0}{d^x}>0$.
Consequently, (strong) variational sufficiency implies noncriticality,
which constitutes a far-reaching extension of the well-known pattern from
nonlinear programming that second-order sufficient conditions rule out
critical multipliers, see \cite{IzmailovSolodov2015}.

For reference, let us record this observation in the following proposition,
which is a direct culmination of the above discussion.

\begin{proposition}\label{prop:VSS_implies_noncriticality}
	If the strong variational sufficient condition
	holds at $\bar x$ for $\bar y\in\overline Y$,
	then $\bar y$ is noncritical for $\bar x$.
\end{proposition}

These observations also suggest that extending the notion of noncriticality via
the graphical derivative, although it leads to the complete characterizations
of \cref{thm:full_characterization_of_error_bound},
might not be the only reasonable choice.
Precisely the lack of robustness discussed in \cref{sec:main_results_discussion},
see \cref{ex:noncritical_limit_of_critical},
makes it delicate to connect noncriticality with superlinear convergence
of Newton-type methods.
Alternatively, noncriticality could be introduced in terms of one of the
robust derivatives, i.e., by the implication
\[
	(0,d^y) \in \mathcal D(\partial \varphi)((\bar x,0),(0,\bar y))(d^x,0)
	\quad\Longrightarrow\quad
	d^x = 0.
\]
Exemplary, choosing $\mathcal D$ to be the limiting graphical derivative
recovers the notion of strong noncriticality
from \cite[Definition~4.6]{BenkoMehlitz2024}.
Unlike with the positive-definiteness criteria \eqref{eq:SOSC_robust}, however,
we expect the resulting notions to genuinely depend on the choice of $\mathcal D$.
Let us mention that such a robust notion of noncriticality might automatically
ensure noncriticality of all neighboring multipliers as well,
a feature closely related to the semi-strict graphical derivative,
a very interesting concept from \cite{HangSarabi2023},
and we leave this direction open for future research.

\section{Consequences in composite optimization}\label{sec:composite_optimization}

In this section, we specify the results from \cref{sec:main} for the
composite optimization problem \eqref{eq:composite_problem}.
We start to present some details on the setting as well as some basic facts in \cref{sec:basics_CP}.
\Cref{sec:gph_subdiff_polyhedral,sec:beyond_polyhedrality} then deepen our insights
for the situations where the outer lsc function $g$ admits a polyhedral subdifferential
mapping or is $C^2$-decomposable, respectively.

\subsection{Setting and basic facts}\label{sec:basics_CP}

Without further mentioning,
we assume that the following assumption is standing throughout the section.

\begin{assumption}\label{ass:composite_problem_basic}
	The functions $f_0\colon\R^n\to\R$ and $F\colon\R^n\to\R^m$ are twice continuously differentiable.
	Furthermore, $g\colon\R^m\to\overline{\R}$ is a proper, lsc function, 
	and $\gph(\partial g)$ is closed.
\end{assumption}

We note that $\gph(\partial g)$ is closed for each proper, lsc function which is continuous
on its domain. However, there exist further practically relevant discontinuous functions 
possessing this property, e.g., all subdifferentially continuous ones, 
see \cite[Definition~13.28]{RockafellarWets1998}.

To start, let us introduce a perturbed version of the model problem \eqref{eq:composite_problem}
by means of
\begin{equation}\label{eq:composite_problem_pert}\tag{CP$(v,u)$}
	\min\limits_x\quad f_0(x) + g(F(x)+u) - \innerprod{v}{x} \quad\text{s.t.}\quad x\in\R^n,
\end{equation}
involving parameters $(v,u)\in\R^n\times\R^m$,
with \eqref{eq:composite_problem} corresponding to \hyperref[eq:composite_problem_pert]{(CP$(0,0)$)}.
This precise approach of perturbing \eqref{eq:composite_problem} dates back 
to the seminal paper \cite{LevyPoliquinRockafellar2000},
and we note that, in the framework of \cref{sec:perturbations}, 
it corresponds to the investigation of \eqref{eq:abstract_problem_pert} 
for $\varphi\colon\R^n\times\R^m\to\overline{\R}$ given by
\[
	\varphi(x,u)
	:=
	f_0(x) + g(F(x)+u).
\]

Throughout the section, we investigate a reference point $((0,0),(\bar x,\bar y))\in\gph M$.
Exploiting \cite[Exercises~8.8 and~10.7]{RockafellarWets1998}, 
we easily find validity of 
$\partial\varphi(x,u)=\{(\nabla_x\LL(x,y),y)\,|\,y\in\partial g(F(x)+u)\}$,
so the stationarity condition $(v,y)\in\partial\varphi(x,u)$ is equivalent to
\begin{equation*}
	\nabla_x\LL(x,y) = v,\quad y \in \partial g(F(x) + u),
\end{equation*}
where, for brevity of notation, 
we made use of the Lagrangian function $\LL$
given in \eqref{eq:Lagrangian}.
Hence, the multiplier mapping $Y\colon\R^n\times\R^m\times\R^n\tto\R^m$ enjoys the representation
\begin{equation}\label{eq:gphY_as_preimage}
	\gph Y
	=
	\{((v,u,x),y)\,|\,(\nabla_x\LL(x,y)-v,(F(x)+u,y))\in \{0\}\times\gph(\partial g)\},
\end{equation}
so, according to \cref{ass:composite_problem_basic}, $\gph Y$ is the preimage of
a closed set under a continuous transformation and, thus, closed.
Recall that $\gph M=\gph Y$, so \eqref{eq:gphY_as_preimage}
also characterizes the mapping $M$ in the present situation.
The following result can be distilled from \cite[Lemma~3.7]{BenkoMehlitz2024}
whose proof,
which just applies the change-of-coordinates formula from \cite[Exercise~6.7]{RockafellarWets1998}
to the representation \eqref{eq:gphY_as_preimage},
does not require convexity of the function $g$ but merely the closedness of $\gph(\partial g)$
which has been assumed in \cref{ass:composite_problem_basic}.

\begin{lemma}\label{lem:graphical_derivative_Y}
	For $(d^v,d^u,d^x,d^y)\in\R^n\times\R^m\times\R^n\times\R^m$,
	the following statements are equivalent:
	\begin{enumerate}
		\item $d^y\in DY((0,0,\bar x),\bar y)(d^v,d^u,d^x)$,
		\item $\nabla^2_{xx}\LL(\bar x,\bar y)d^x+F'(\bar x)^\top d^y = d^v$,
			$d^y\in D(\partial g)(F(\bar x),\bar y)(F'(\bar x)d^x+d^u)$. 
	\end{enumerate}
\end{lemma}

According to our considerations in \cref{sec:perturbations} and \cref{lem:graphical_derivative_Y}, 
$\bar y\in\overline Y$ is a noncritical multiplier for $\bar x$ in the sense of \cref{def:criticality}
if and only if 
\[
	\nabla^2_{xx}\LL(\bar x,\bar y)d^x + F'(\bar x)^\top d^y=0,\,
		d^y \in D(\partial g)(F(\bar x),\bar y)(F'(\bar x)d^x)
	\quad\Longrightarrow\quad
	d^x = 0,
\]
and this is consistent with the notion of noncriticality introduced in
\cite[Definition~4.1]{BenkoMehlitz2024}, see \cite[Definition~3.1]{MordukhovichSarabi2018} as well.
It has been illustrated in \cite[Section~4]{BenkoMehlitz2024} that 
\cref{def:criticality} covers divers notions of noncriticality coined for more
specified problems than \eqref{eq:composite_problem} like, exemplary,
standard nonlinear optimization problems with equality and/or inequality constraints,
see \cite[Definition~2]{IzmailovSolodov2012}.

Let us also note that, due to \cref{lem:graphical_derivative_Y} again, 
condition \eqref{eq:abstract_uniqueness_condition} takes the form
\[
	F'(\bar x)^\top d^y = 0,\,d^y\in D(\partial g)(F(\bar x),\bar y)(0)
	\quad\Longrightarrow\quad
	d^y=0
\]
in the present situation.
As already mentioned before,
the latter characterizes the isolated calmness of $Y$ (of $Y_{\bar x}$) 
at $((0,0,\bar x),\bar y)$ (at $((0,0),\bar y)$) and, particularly,
implies uniqueness of the multiplier $\bar y$,
see \cite[Corollary~3.8]{BenkoMehlitz2024} and the recent study
\cite{MehlitzWachsmuth2026} for details and further information.

The upcoming lemma presents equivalent characterizations of the
calmness-type properties of the mapping $M$ from \cref{def:semi_isolated_calmness_M1}
for the situation at hand.

\begin{lemma}\label{lem:semi_isolated_calmness_as_error_bound}
	\hfill
	\begin{enumerate}
	\item
	Mapping $M$ is partially isolatedly calm w.r.t.\ $x$ at $((0,0),(\bar x,\bar y))$
	if and only if there are a neighborhood 
	$\NN\subset\R^n\times\R^m$ of $(\bar x,\bar y)$ and a constant $\kappa>0$ such that 
	\begin{equation}\label{eq:Isol_Calm_M_estimate_via_y}
		\forall (x,y)\in \NN\colon\quad
		\norm{x-\bar x}\leq\kappa\bigl(\norm{\nabla_x\LL(x,y)} + \dist(F(x),(\partial g)^{-1}(y))\bigr).
	\end{equation}
	\item
	Mapping $M$ is semi-isolatedly calm w.r.t.\ $x$ at $((0,0),(\bar x,\bar y))$
	if and only if there are a neighborhood 
	$\NN\subset\R^n\times\R^m$ of $(\bar x,\bar y)$ and a constant $\kappa>0$ such that 
	\begin{equation}\label{eq:Isol_Calm_M_estimate_via_y_special}
		\forall (x,y)\in \NN\colon\quad
		\norm{x-\bar x} + \dist(y,\overline Y)
		\leq
		\kappa\bigl(\norm{\nabla_x\LL(x,y)} + \dist(F(x),(\partial g)^{-1}(y))\bigr).
	\end{equation}
	\end{enumerate}
\end{lemma}
\begin{proof}
	We only prove the first assertion as the second one can be shown in analogous fashion.
	
	Let $M$ be partially isolatedly calm w.r.t.\ $x$ at the reference point $((0,0),(\bar x,\bar y))$.
	Then we find neighborhoods $\widetilde \PP\subset\R^n\times\R^m$ of $(0,0)$,
	$\widetilde\XX\subset\R^n$ of $\bar x$, and $\widetilde\YY\subset\R^m$ of $\bar y$
	as well as a constant
	$\kappa>0$ such that
	\[
		\forall (v,u)\in\widetilde \PP,\,
		\forall (x,y)\in M(v,u)\cap(\widetilde \XX\times\widetilde\YY)\colon\quad
		\norm{x-\bar x}\leq\kappa(\norm{u}+\norm{v}).
	\]
	Further, pick some arbitrary $\delta>0$ such that 
	$\{(v,u)\,|\,\norm{u}+\norm{v}\leq\delta\}\subset \widetilde \PP$
	holds, 
	and define $\XX:=\widetilde \XX\cap\{x\,|\,\norm{x-\bar x}\leq\delta/(\kappa^{-1}+1)\}$.
	Set $\NN:=\XX\times \widetilde \YY$ and pick $(x,y)\in \NN$.
	If $(\partial g)^{-1}(y)=\emptyset$, then \eqref{eq:Isol_Calm_M_estimate_via_y} holds
	trivially.
	Thus, let us assume that $(\partial g)^{-1}(y)\neq\emptyset$.
	Set $v:=\nabla_x\LL(x,y)$ and let $u\in\R^m$ be a point
	such that $F(x)+u\in(\partial g)^{-1}(y)$ and $\norm{u}=\dist(F(x),(\partial g)^{-1}(y))$, i.e.,
	$u$ is a projection of $0$ onto the nonempty, closed set $(\partial g)^{-1}(y)-F(x)$.
	If $\norm{u}+\norm{v}\leq\delta$, then $(v,u)\in\widetilde \PP$ and $(x,y)\in M(v,u)\cap(\widetilde\XX\times\widetilde\YY)$,
	so $\norm{x-\bar x}\leq\kappa(\norm{u}+\norm{v})$.
	Otherwise, the same estimate follows by definition of $\XX$.
	Hence, by definition of $u$ and $v$, \eqref{eq:Isol_Calm_M_estimate_via_y} holds.
	
	Let us assume that there are a neighborhood $\NN\subset\R^n\times\R^m$ of $(\bar x,\bar y)$
	and a constant $\kappa>0$ such that \eqref{eq:Isol_Calm_M_estimate_via_y} is valid.
	Choose neighborhoods $\XX\subset\R^n$ of $\bar x$ and $\YY\subset\R^m$ of $\bar y$
	such that $\XX\times\YY\subset\NN$.
	Pick $(v,u)\in\R^n\times\R^m$ and $(x,y)\in M(v,u)\cap (\XX\times\YY)$ arbitrarily.
	By definition of $M$, this yields $v=\nabla_x\LL(x,y)$ and $F(x)+u\in(\partial g)^{-1}(y)$, i.e.,
	\[
		\norm{x-\bar x}
		\leq
		\kappa\bigl(\norm{\nabla_x\LL(x,y)}+\dist(F(x),(\partial g)^{-1}(y))\bigr)
		\leq
		\kappa(\norm{v}+\norm{u}).
	\]
	Hence, $M$ is partially isolatedly calm w.r.t.\ $x$ at $((0,0),(\bar x,\bar y))$.
\end{proof}

According to \cref{lem:relations_calmness_x,lem:semi_isolated_calmness_as_error_bound},
validity of the error bound \eqref{eq:Isol_Calm_M_estimate_via_y_special} is equivalent to the validity
of the error bound \eqref{eq:Isol_Calm_M_estimate_via_y} 
(for different $\NN$ and $\kappa$, potentially)
combined with the calmness of the multiplier mapping $Y$ at $((0,0,\bar x),\bar y)$.
Recall from \cref{sec:main_results_discussion} that the latter is closely tied to
the calmness of the restricted multiplier mapping $Y_{\bar x}$ at $((0,0),\bar y)$,
on which we elaborate in the subsequently stated remark.

\begin{remark}\label{rem:calmness_of_Yx}
Due to
\begin{equation*}
	\gph Y_{\bar x}
	=
	\{
		((v,u),y)\,|\,\nabla_x \LL(\bar x,y)-v=0,\,(F(\bar x)+u,y)\in \gph(\partial g)
	\},
\end{equation*}
it is easy to check that $Y_{\bar x}$ is a polyhedral mapping, and, thus,
calm at each point of its graph,
whenever $\gph(\partial g)$ is a polyhedral set,
see \cite[Proposition~1]{Robinson1981}.
Exemplary, this happens to be the case whenever $\epi g$ is polyhedral.
Indeed, \cite[Lemma~2.6]{BenkoMehlitz2025} shows that the graph of the
limiting normal cone mapping associated with $\epi g$ is polyhedral
in this situation, and $\gph(\partial g)$ is obtained from
that graph by computing its intersection with the polyhedral set $\R^m\times\R\times\R^m\times\{-1\}$.
Hence, this observation is a first ingredient in explaining why there is no additional
calmness-type assumption stated in \cite[Proposition~1]{IzmailovSolodov2012}.
It is automatically satisfied in the setting of equality-constrained optimization
where $g=\delta_{\{0\}}$
and, thus, $\epi g=\{0\}\times\R_+$ and $\gph(\partial g)=\{0\}\times\R^m$;
a full explanation is provided in \cref{sec:gph_subdiff_polyhedral} below.
\end{remark}

Complementing \cref{rem:calmness_of_Yx},
let us now clarify the relation between the calmness properties of $Y$ and $Y_{\bar x}$
in the composite setting.
In view of \cref{lem:parametric_calmness_transfer},
the only missing piece is the parametric calmness of $M$ w.r.t.\ $x$,
which we verify next --- remarkably, without any assumptions on the function $g$ whatsoever.

\begin{proposition}\label{prop:parametric_calmness_composite}
	In the composite setting, i.e., for $M$ given by \eqref{eq:gphY_as_preimage}
	and noting that $\gph M = \gph Y$,
	mapping $M$ is parametrically calm w.r.t.\ $x$ at $((0,0),(\bar x,\bar y))$.
	Moreover, the neighborhood $\PP$ in \eqref{eq:parametric_calmness} can be chosen
	as the whole space $\R^n\times\R^m$.
\end{proposition}
\begin{proof}
	By \eqref{eq:gphY_as_preimage}, the parameter mapping $Q_y\colon\R^n\tto\R^n\times\R^m$
	associated with $M$ attains, for each $y\in\R^m$, the form
	\[
		Q_y(x)
		=
		\{\nabla_x\LL(x,y)\}\times\bigl((\partial g)^{-1}(y)-F(x)\bigr).
	\]
	Fix $\delta>0$ and a bounded neighborhood $\XX\subset\R^n$ of $\bar x$ such that
	$\nabla f_0$, $F$, and $F'$ are Lipschitz continuous on $\XX$
	with a common constant $\rho>0$, and let $\YY$ be the closed ball around $\bar y$
	with radius $\delta$.
	Pick $x\in\XX$, $y\in\YY$, and $(v,u)\in Q_y(x)$,
	i.e., $v=\nabla_x\LL(x,y)$ and $F(x)+u\in(\partial g)^{-1}(y)$.
	Then $v':=\nabla_x\LL(\bar x,y)$ and $u':=u+F(x)-F(\bar x)$ satisfy
	$(v',u')\in Q_y(\bar x)$ due to $F(\bar x)+u'=F(x)+u\in(\partial g)^{-1}(y)$.
	Furthermore, exploiting $\nabla_x\LL(x,y)=\nabla f_0(x)+F'(x)^\top y$, we obtain
	\[
		\norm{v'-v}
		\leq
		\norm{\nabla f_0(\bar x)-\nabla f_0(x)}+\norm{F'(\bar x)-F'(x)}\norm{y}
		\leq
		\rho\bigl(1+\norm{\bar y}+\delta\bigr)\norm{x-\bar x}
	\]
	as well as $\norm{u'-u}=\norm{F(x)-F(\bar x)}\leq\rho\norm{x-\bar x}$.
	Hence, \eqref{eq:parametric_calmness} holds with $\PP:=\R^n\times\R^m$ and
	$\kappa:=\rho(2+\norm{\bar y}+\delta)$.
\end{proof}

Let us point out that the proof above uses only the smoothness of $f_0$ and $F$
imposed in \cref{ass:composite_problem_basic}.
Alternatively, \cref{prop:parametric_calmness_composite} can be derived from
\cref{rem:ample_parametrization} since, in the representation \eqref{eq:gphY_as_preimage},
the partial Jacobian of the smooth part w.r.t.\ $(v,u)$ is even a bijection;
the direct proof above, however, provides explicit constants.
Together with \cref{lem:parametric_calmness_transfer},
\cref{prop:parametric_calmness_composite} yields the following corollary.

\begin{corollary}\label{cor:calmness_Y_equivalence_composite}
	In the composite setting, the multiplier mapping $Y$ is calm at $((0,0,\bar x),\bar y)$
	if and only if the restricted multiplier mapping $Y_{\bar x}$ is calm at $((0,0),\bar y)$.
\end{corollary}
\begin{proof}
	Calmness of $Y$ always implies calmness of $Y_{\bar x}$
	as pointed out in \cref{sec:main_results_discussion},
	and the converse implication follows by combining
	\cref{lem:parametric_calmness_transfer} with \cref{prop:parametric_calmness_composite}.
\end{proof}

\cref{thm:full_characterization_of_error_bound}, combined with
\cref{lem:graphical_derivative_Y,lem:semi_isolated_calmness_as_error_bound},
fully characterizes the error bounds \eqref{eq:Isol_Calm_M_estimate_via_y} and \eqref{eq:Isol_Calm_M_estimate_via_y_special} in terms of noncriticality of the
multiplier $\bar y$ for $\bar x$ and the calmness-type conditions
\ref{item:weak_inner_calmness} and \ref{item:calmness}
imposed on the multiplier mapping $Y$.
In the upcoming \cref{sec:gph_subdiff_polyhedral},
we show that these additional conditions are inherently satisfied if $\gph(\partial g)$ is a polyhedral set,
and \cref{sec:beyond_polyhedrality} secures them far beyond polyhedrality.

\subsection{Polyhedrality of the graph of the subdifferential mapping}\label{sec:gph_subdiff_polyhedral}

In this subsection, we show that the calmness-type conditions \ref{item:weak_inner_calmness}
and \ref{item:calmness} of \cref{thm:full_characterization_of_error_bound}
are inherently satisfied in the composite setting whenever $\gph(\partial g)$ is a polyhedral set.

Let us start to address condition \ref{item:weak_inner_calmness}.
As it is concerned with the completely new notion of weak inner calmness in the fuzzy sense, 
the next result and its proof are also novel.

\begin{proposition}\label{prop:gph_of_subdifferential_polyhedral}
	Let $\gph(\partial g)$ be polyhedral.
	Then the mapping $Y$ is weakly inner calm in the fuzzy sense at $((0,0,\bar x),\bar y)$
	in each direction $(d^v,d^u,d^x)\neq(0,0,0)$,
	i.e., condition \ref{item:weak_inner_calmness} 
	of \cref{thm:full_characterization_of_error_bound} holds.
\end{proposition}
\begin{proof}
	Pick some arbitrary direction $(d^v,d^u,d^x)\neq(0,0,0)$.
	We assume that there are sequences $t_k\downarrow 0$ and
	$(d^v_k,d^u_k,d^x_k)\to(d^v,d^u,d^x)$ as well as $y_k\to \bar y$
	such that $y_k\in Y(v_k,u_k,x_k)$ for each $k\in\N$ where
	$(v_k,u_k,x_k):=(t_kd^v_k,t_kd^u_k,\bar x+t_kd^x_k)$.
	By definition, this gives $\nabla_x\LL(x_k,y_k)=v_k$ and
	$(F(x_k)+u_k,y_k)\in\gph(\partial g)$ for each $k\in\N$,
	and
	$\nabla_x\LL (\bar x,\bar y)=0$ as well as $(F(\bar x),\bar y)\in\gph(\partial g)$
	hold by definition of the multiplier set.
	A first-order Taylor expansion gives
	\begin{align*}
		v_k
		&=
		\nabla_x\LL (x_k,y_k)-\nabla_x\LL (\bar x,\bar y)
		\\
		&=
		\nabla_x\LL (\bar x,y_k)+\nabla^2_{xx}\LL (\bar x,y_k)(x_k-\bar x)+\oo(\nnorm{x_k-\bar x})-\nabla_x\LL (\bar x,\bar y)
		\\
		&=
		F'(\bar x)^\top(y_k-\bar y)+t_k\nabla^2_{xx}\LL (\bar x,y_k)d^x_k+\oo(t_k)
		\\
		&=
		F'(\bar x)^\top(y_k-\bar y)+t_k\nabla^2_{xx}\LL (\bar x,\bar y)d^x_k+\oo(t_k)
	\end{align*}
	and, thus,
	\begin{equation}\label{eq:some_bound_from_Taylor}
		F'(\bar x)^\top(y_k-\bar y)/t_k
		=
		d^v_k - \nabla^2_{xx}\LL (\bar x,\bar y)d^x_k + \oo(1).
	\end{equation}	
	Applying \cite[Proposition~8.24]{Ioffe2017}, which is possible
	since $\gph(\partial g)$ is polyhedral, we find a neighborhood $\OON\subset \R^m\times\R^m$
	of the origin such that
	\begin{equation}\label{eq:exactness_tangential_approximation}
		\bigl(\gph(\partial g)-(F(\bar x),\bar y)\bigr)\cap\OON
		=
		\gph D(\partial g)(F(\bar x),\bar y)\cap\OON.
	\end{equation}	
	Hence, the relation
	\begin{equation*}
		\bigl(F(x_k)+u_k-F(\bar x),y_k-\bar y\bigr)
		\in 
		\gph D(\partial g)(F(\bar x),\bar y)
	\end{equation*}
	is obtained for all sufficiently large $k\in\N$.
	Noting that $\gph D(\partial g)(F(\bar x),\bar y)$ is a cone, we already have
	\[
		\bigl((F(x_k)-F(\bar x))/t_k+d^u_k,(y_k-\bar y)/t_k\bigr)
		\in 
		\gph D(\partial g)(F(\bar x),\bar y)
	\]
	for all sufficiently large $k\in\N$.
	As $\gph(\partial g)$ is a polyhedral set, the same holds true 
	for $\gph D(\partial g)(F(\bar x),\bar y)$.
	Indeed, this follows from the union rule for tangents,
	see e.g.\ \cite[Table~4.1]{AubinFrankowska2009},
	and the well-known fact that tangent cones to convex polyhedral sets are convex polyhedral.
	Hence, we find a convex polyhedral set $C\subset\gph D(\partial g)(F(\bar x),\bar y)$
	such that along a subsequence (without relabeling)
	\[
		\bigl((F(x_k)-F(\bar x))/t_k+d^u_k,(y_k-\bar y)/t_k\bigr)
		\in C
	\]
	holds for all $k\in\N$.
	Clearly, there exist matrices $A$, $B$, and $c$ of appropriate
	dimensions such that $C=\{(z,w)\,|\,Az+Bw\leq c\}$, i.e.,
	\[
		A\bigl((F(x_k)-F(\bar x))/t_k+d^u_k\bigr)
		+
		B(y_k-\bar y)/t_k
		\leq c
	\]
	holds for all $k\in\N$.
	Observe that the linear system
	\begin{equation}\label{eq:linear_system}
		B\eta \leq c - A\bigl((F(x_k)-F(\bar x))/t_k+d^u_k\bigr),
		\quad
		F'(\bar x)^\top \eta = F'(\bar x)^\top(y_k-\bar y)/t_k
	\end{equation}
	possesses the solution $\eta=(y_k-\bar y)/t_k$ for each $k\in\N$.
	Hence, for each $k\in\N$,
	we can apply Hoffman's lemma, see e.g.\ \cite[Theorem~8.33]{Ioffe2017},
	to find a solution $\eta_k\in\R^m$ of \eqref{eq:linear_system}
	which additionally satisfies
	\[
		\norm{\eta_k}
		\leq
		\kappa\bigl(
			\nnorm{c - A((F(x_k)-F(\bar x))/t_k+d^u_k)}
			+
			\nnorm{F'(\bar x)^\top(y_k-\bar y)/t_k}
		\bigr)
	\]
	for some constant $\kappa>0$ which does not depend on $k\in\N$. 
	Together with \eqref{eq:some_bound_from_Taylor},
	this estimate shows that $\eta_k$ is bounded,
	say $\norm{\eta_k}\leq \Theta$ hold for all $k\in\N$ and some constant $\Theta>0$.
	Furthermore, we find 
	\[
		\bigl((F(x_k)-F(\bar x))/t_k+d^u_k,\eta_k\bigr)\in C,
	\]	
	and, consequently,
	\[
		\bigl(F(x_k)+u_k-F(\bar x),t_k\eta_k\bigr)\in\gph D(\partial g)(F(\bar x),\bar y)
	\]
	for all $k\in\N$.
	Applying \eqref{eq:exactness_tangential_approximation} once more,
	$(F(x_k)+u_k,\bar y+t_k\eta_k)\in\gph(\partial g)$, i.e.,
	$\bar y+t_k\eta_k\in\partial g(F(x_k)+u_k)$
	follows for all sufficiently large $k\in\N$.
	For each sufficiently large $k\in\N$,
	setting $\tilde y_k:=\bar y+t_k\eta_k$ and $\tilde v_k:=\nabla f_0(x_k)+F'(x_k)^\top \tilde y_k$ yields
	$\tilde y_k\in\partial g(F(x_k)+u_k)$ and $\tilde y_k\in Y(\tilde v_k,u_k,x_k)$.
	Furthermore, we have
	\begin{align*}
		(v_k-\tilde v_k)/t_k
		&=
		F'(x_k)^\top(y_k-\tilde y_k)/t_k
		=
		F'(x_k)^\top(y_k-\bar y)/t_k-F'(x_k)^\top\eta_k
		\\
		&=
		F'(\bar x)^\top(y_k-\bar y)/t_k + (F'(x_k)-F'(\bar x))^\top(y_k-\bar y)/t_k
		\\
		&\qquad
		-F'(\bar x)^\top\eta_k-(F'(x_k)-F'(\bar x))^\top\eta_k
		\\
		&=
		\frac{(F'(x_k)-F'(\bar x))^\top}{t_k}(y_k-\bar y)
		-
		(F'(x_k)-F'(\bar x))^\top\eta_k
		\to
		0
	\end{align*}
	by construction of $\eta_k$ and
	boundedness of $(F'(x_k)-F'(\bar x))^\top/t_k$ as well as $\eta_k$, 
	and this shows that $\tilde d^v_k:=\tilde v_k/t_k\to d^v$.
	Finally, we have
	\begin{align*}
		\nnorm{\tilde y_k-\bar y}
		=
		t_k\norm{\eta_k}
		&=
		\frac{\norm{\eta_k}}{\nnorm{(\tilde d^v_k,d^u_k,d^x_k)}}
			\nnorm{(\tilde v_k,u_k,x_k)-(0,0,\bar x)}
		\\
		&\leq
		\frac{2\,\Theta}{\norm{(d^v,d^u,d^x)}}t_k\nnorm{(\tilde d^v_k,d^u_k,d^x_k)}
	\end{align*}
	for all large enough $k\in\N$.
	Thus, $Y$ is weakly inner calm in the fuzzy sense at $((0,0,\bar x),\bar y)$
	in direction $(d^v,d^u,d^x)$.
\end{proof}

Next, let us address condition \ref{item:calmness},
which requires no extra effort in view of the preparations from \cref{sec:basics_CP}.

\begin{proposition}\label{prop:calmness_Y_composite_polyhedral}
	Let $\gph(\partial g)$ be polyhedral.
	Then $Y$ is calm at $((0,0,\bar x),\bar y)$,
	i.e., condition \ref{item:calmness} of \cref{thm:full_characterization_of_error_bound} holds.
\end{proposition}
\begin{proof}
	By \cref{rem:calmness_of_Yx}, $Y_{\bar x}$ is calm at $((0,0),\bar y)$,
	so the claim follows from \cref{cor:calmness_Y_equivalence_composite}.
\end{proof}

Let us now formulate a result which combines our findings so far;
it can be viewed as the composite counterpart of \cref{thm:full_characterization_of_error_bound}
for the special situation considered in this subsection.

\begin{theorem}\label{thm:gph_of_subdifferential_polyhedral}
	Let $\gph(\partial g)$ be polyhedral.
	Then the following statements are equivalent.
	\begin{enumerate}
		\item The multiplier $\bar y$ is noncritical for $\bar x$.
		\item There are a neighborhood $\NN\subset\R^n\times\R^m$ of $(\bar x,\bar y)$
			and a constant $\kappa>0$ such that \eqref{eq:Isol_Calm_M_estimate_via_y}
			holds for all $(x,y)\in \NN$.
		\item There are a neighborhood $\NN\subset\R^n\times\R^m$ of $(\bar x,\bar y)$
			and a constant $\kappa>0$ such that \eqref{eq:Isol_Calm_M_estimate_via_y_special}
			holds for all $(x,y)\in \NN$.
	\end{enumerate}
\end{theorem}
\begin{proof}
	Due to \cref{lem:semi_isolated_calmness_as_error_bound},
	the second and the third statement correspond to the partial isolated calmness
	and the semi-isolated calmness of $M$ w.r.t.\ $x$ at $((0,0),(\bar x,\bar y))$,
	respectively.
	The claim, thus, follows from \cref{thm:full_characterization_of_error_bound}
	since its conditions \ref{item:weak_inner_calmness} and \ref{item:calmness}
	are valid by 
	\cref{prop:gph_of_subdifferential_polyhedral,prop:calmness_Y_composite_polyhedral},
	respectively.
\end{proof}

In \cite{Sarabi2022}, the author considers the particular situation where the
function $g$ is convex piecewise linear-quadratic.
This means that $g$ is a convex function whose domain can be decomposed into
finitely many convex polyhedral sets such that $g$ is a quadratic function
on each of these sets. Exemplary, observe that each convex piecewise linear function,
i.e., each function whose epigraph is convex polyhedral, is
convex piecewise linear-quadratic.
It follows from \cite[Proposition~2.1(a)]{Sarabi2022} that each convex piecewise
linear-quadratic function comes along with a polyhedral limiting subdifferential mapping.
Hence, \cref{thm:gph_of_subdifferential_polyhedral} 
applies to this situation.

Let us assume that $g$ is convex piecewise linear-quadratic.
In \cite[Theorem~3.6]{Sarabi2022}, the author shows that
the following statements are equivalent.
\begin{enumerate}
	\item\label{item:MorSar_noncritical} 
		The multiplier $\bar y$ is noncritical for $\bar x$.
	\item\label{item:MorSar_semi_isolated_calmness_M1}
		Mapping $M$ is semi-isolatedly calm w.r.t.\ $x$ at $((0,0),(\bar x,\bar y))$.
\end{enumerate}
For a certain subclass of convex piecewise linear-quadratic functions,
this result can also be found in \cite[Theorem~5.1]{DoMordukhovichSarabi2021}.
Furthermore, the same characterization is given in \cite[Theorem~4.1]{MordukhovichSarabi2018}
for the even smaller class of composite problems where $\epi g$ is convex polyhedral.
All these findings generalize the classical characterization of
noncritical multipliers in smooth equality-constrained optimization
from \cite[Proposition~1]{IzmailovSolodov2012}.
Consulting \cref{lem:semi_isolated_calmness_as_error_bound} yields that
condition~\ref{item:MorSar_semi_isolated_calmness_M1} from above
is equivalent to the existence of a neighborhood $\NN\subset\R^n\times\R^m$
of $(\bar x,\bar y)$ and some constant $\kappa>0$ such that
\eqref{eq:Isol_Calm_M_estimate_via_y_special} holds for all $(x,y)\in \NN$.
Thus, \cite[Theorem~3.6]{Sarabi2022} matches \cref{thm:gph_of_subdifferential_polyhedral}
for convex piecewise linear-quadratic functions $g$. 

Let us close this subsection by mentioning some practically relevant
examples of nonconvex functions $g$
such that $\gph(\partial g)$ is polyhedral.
If $g$ is the indicator function of a nonconvex polyhedral set,
then \eqref{eq:composite_problem} represents so-called disjunctive optimization problems, 
see \cite{BenkoCervinkaHoheisel2019,FlegelKanzowOutrata2007,Mehlitz2019a} for an introduction,			
like mathematical programs with so-called
complementarity, switching, or vanishing constraints.
As mentioned in \cref{rem:calmness_of_Yx}, $\partial g$ is a polyhedral mapping in this case.
Another popular setting is recovered 
whenever the role of $g$ is played by the $\ell_0$-quasi-norm given by
$g(z):=|\{i\in\{1,\ldots,m\}\,|\,z_i\neq 0\}|$ for all $z\in\R^m$.
This function is popular in sparse optimization since its slope
at vectors possessing a vanishing component is infinite.
Besides being proper and lsc, 
$g$ lacks most of the desirable variational properties 
like convexity and continuity.
However, simple calculations yield
$\partial g(z)=\{y\,|\,\forall i\in I^\pm(z)\colon\,y_i=0\}$
for each $z\in\R^m$, where $I^\pm(z):=\{i\in\{1,\ldots,m\}\,|\,z_i\neq 0\}$,
and it is not hard to see from this formula that
$\partial g$ is a polyhedral mapping in this situation. 
Hence, \cref{thm:gph_of_subdifferential_polyhedral} applies.

\subsection{Going beyond polyhedrality}\label{sec:beyond_polyhedrality}

Without polyhedrality of $\gph(\partial g)$, the verification of the conditions
\ref{item:weak_inner_calmness} and \ref{item:calmness} of
\cref{thm:full_characterization_of_error_bound} becomes much more challenging.
In this subsection, we address this issue for the setting where
$g$ is $C^2$-decomposable,
see \cite{Shapiro2003}, where this class has been introduced,
as well as \cite{HangSarabi2023,Milzarek2016}.
The class is rich.
On the one hand, it comprises the indicator functions of $C^2$-cone reducible sets
in the sense of \cite[Definition~3.135]{BonnansShapiro2000},
which form the setting of \cite{MordukhovichSarabi2019} and cover
prominent examples like polyhedral sets, the second-order cone,
and the cone of all positive semidefinite matrices.
On the other hand, the extension is a genuine one, as
$C^2$-decomposability does not presuppose any constraint structure:
polyhedral convex functions such as the $\ell_1$-norm
as well as important spectral functions
like the maximum eigenvalue of a symmetric matrix
belong to the class, too,
see \cite[Section~5.3.3]{Milzarek2016} and \cite[Examples~2.1--2.4]{Shapiro2003}.

The study of noncriticality in nonpolyhedral settings was pioneered in \cite{MordukhovichSarabi2019}.
In \cite[Theorem~5.6]{MordukhovichSarabi2019}, it is shown that, for $g:=\delta_D$
with a set $D$ which is $C^2$-cone reducible at $F(\bar x)$, noncriticality
of $\bar y$ for $\bar x$ is equivalent to the semi-isolated calmness of
$M$ w.r.t.\ $x$ at $((0,0),(\bar x,\bar y))$ provided that
$Y_{\bar x}$ is calm at $((0,0),\bar y)$ and that a certain closedness
condition is satisfied.
Related error bound results for $C^2$-decomposable composite problems,
obtained under a second-order sufficient condition and, again, calmness of
$Y_{\bar x}$, can be found in \cite[Theorem~2]{HangSarabi2023}.
Notably, no closedness condition is needed there.
This is achieved by employing the semi-strict graphical derivative,
a novel and very interesting concept introduced in \cite{HangSarabi2023},
and clearly shows that the route pursued below is not the only viable one.
The main result of this subsection, \cref{thm:decomposable_characterization},
provides a complete characterization the error
bounds \eqref{eq:Isol_Calm_M_estimate_via_y} and \eqref{eq:Isol_Calm_M_estimate_via_y_special}
in terms of noncriticality
for the entire class of $C^2$-decomposable functions, under two
assumptions in the spirit of \cite{MordukhovichSarabi2019},
namely calmness of $Y_{\bar x}$ and a closedness condition.
The route goes through \cref{thm:full_characterization_of_error_bound}.
The calmness of $Y_{\bar x}$ settles condition \ref{item:calmness} via
\cref{cor:calmness_Y_equivalence_composite}, and the crucial step,
\cref{prop:weak_inner_calmness_decomposable}, verifies condition
\ref{item:weak_inner_calmness}.
Recall that both conditions are also necessary for semi-isolated
calmness, so this route is, in a sense, the canonical one.
In particular, we fully recover the characterization from
\cite[Theorem~5.6]{MordukhovichSarabi2019}, enriched by the counterpart in
terms of the error bound \eqref{eq:Isol_Calm_M_estimate_via_y},
as the special case treated in \cref{sec:C2_cone_reducible}.

Let us emphasize that our proof strategy essentially follows the one of
\cite[Theorem~5.6]{MordukhovichSarabi2019}, including the following elegant
trick: while, in polyhedral settings, one measures the deviation of a
multiplier sequence directly from the reference multiplier $\bar y$ and
controls it via Hoffman's lemma, see our proof of
\cref{prop:gph_of_subdifferential_polyhedral}, the deviation is now measured
from a moving projection onto $\overline Y$, made available at the
correct rate precisely by the imposed calmness of $Y_{\bar x}$, and this
bypasses the absence of polyhedrality.
We merely extend these arguments from indicator functions of $C^2$-cone
reducible sets to $C^2$-decomposable functions, for which the required
second-order toolkit has been provided in \cite{HangSarabi2023}.
Our approach, however, also brings a simplification:
the calmness of $Y_{\bar x}$ is exploited directly as an error bound in the
original problem data, which renders the propagation results
\cite[Lemmas~5.3, 5.4, and~5.5]{MordukhovichSarabi2019} between the original
and the reduced system superfluous,
and the same applies to \cite[Lemma~5.2]{MordukhovichSarabi2019}
in view of \cref{rem:closedness_representation_free}.

\subsubsection{$C^2$-decomposable functions}\label{sec:decomposable}

Following \cite{Shapiro2003}, see also \cite[Section~3]{HangSarabi2023},
we say that function $g$ is $C^2$-decomposable at $\bar z:=F(\bar x)$ if there are an
open neighborhood $\VV\subset\R^m$ of $\bar z$, 
a proper, lsc, sublinear function $\theta\colon\R^p\to\overline\R$,
and a twice continuously differentiable mapping $\Xi\colon\R^m\to\R^p$
such that $\Xi(\bar z)=0$ and
\begin{equation}\label{eq:decomposition}
	\forall z\in\VV\colon\quad
	g(z) = g(\bar z) + \theta(\Xi(z)).
\end{equation}
Recall that sublinearity means convexity together with positive homogeneity,
see \cite[Definition~3.18]{RockafellarWets1998}, and note that it entails
$\theta(0)=0$ since $\theta$ is proper and lsc,
see \cite[Exercise~3.19]{RockafellarWets1998}.
By \cite[Theorem~8.24]{RockafellarWets1998}, $\theta$ is precisely the support
function of 
$\partial\theta(0)=\{\mu\,|\,\forall q\in\R^p\colon\,\innerprod{\mu}{q}\leq\theta(q)\}$
which is a nonempty, closed, convex set.
Moreover, for each $q\in\dom\theta$, the set $\partial\theta(q)$ consists precisely
of the maximizers of $\innerprod{\cdot}{q}$ over $\partial\theta(0)$,
see \cite[Corollary~8.25]{RockafellarWets1998},
and, as the associated maximum equals $\theta(q)$, this reads
\begin{equation}\label{eq:subdiff_sublinear}
	\forall q\in\R^p\colon\quad
	\partial\theta(q)
	=
	\{\mu\in\partial\theta(0)\,|\,\innerprod{\mu}{q}=\theta(q)\},
\end{equation}
where both sides are empty whenever $q\notin\dom\theta$.
Denoting by $\HH\subset\R^p$ the linear subspace parallel to the affine
hull of $\partial\theta(0)$, we assume, in addition, that the
nondegeneracy condition
\begin{equation}\label{eq:nondegeneracy}
	\HH\cap\ker\Xi'(\bar z)^\top=\{0\}
\end{equation}
is satisfied, see \cite[equation (32)]{HangSarabi2023}.
Let us point out that \eqref{eq:nondegeneracy} is clearly implied by
surjectivity of $\Xi'(\bar z)$, which annihilates $\ker\Xi'(\bar z)^\top$,
but is considerably weaker in general.
Throughout this subsection, we take the $C^2$-decomposability of $g$ at
$\bar z$, with a fixed representation \eqref{eq:decomposition}, together with
the nondegeneracy condition \eqref{eq:nondegeneracy} as standing assumptions.

The following lemma clarifies that, in this setting, $\partial g$ is fully
described by well-behaved reduced multipliers.

\begin{lemma}\label{lem:decomposable_multipliers}
	There are an open neighborhood $\WW\subset\VV$ of $\bar z$
	and a constant $c>0$ such that the following statements hold.
	\begin{enumerate}
		\item\label{item:dec_injectivity}
			For all $z\in\WW$ and all $\nu\in\HH$,
			we have $\nnorm{\Xi'(z)^\top\nu}\geq c\norm{\nu}$.
		\item\label{item:dec_chain_rule}
			For each $z\in\WW$, we have
			$\partial g(z)=\Xi'(z)^\top\partial\theta(\Xi(z))$,
			and for each $y\in\partial g(z)$,
			there is a unique multiplier $\mu(z,y)\in\partial\theta(\Xi(z))$
			satisfying $y=\Xi'(z)^\top\mu(z,y)$.
		\item\label{item:dec_multiplier_convergence}
			Whenever sequences $z_k\to\bar z$ and $y_k\to\hat y$ satisfy
			$z_k\in\WW$ and $y_k\in\partial g(z_k)$ for all $k\in\N$,
			we have $\hat y\in\partial g(\bar z)$ and
			$\mu(z_k,y_k)\to\mu(\bar z,\hat y)$.
	\end{enumerate}
\end{lemma}
\begin{proof}
	As \eqref{eq:nondegeneracy} states that the linear operator
	$\nu\mapsto\Xi'(\bar z)^\top\nu$ is injective on the subspace $\HH$,
	there is a constant $c>0$ satisfying
	$\nnorm{\Xi'(\bar z)^\top\nu}\geq 2c\norm{\nu}$ for all $\nu\in\HH$,
	and the continuity of $\Xi'$ yields an open neighborhood
	$\WW\subset\VV$ of $\bar z$ such that
	item~\ref{item:dec_injectivity} holds.

	To verify item~\ref{item:dec_chain_rule}, we fix $z\in\WW$.
	As $\theta$ is convex,
	the chain rules from \cite[Theorem~10.6]{RockafellarWets1998} yield
	$\partial g(z)=\Xi'(z)^\top\partial\theta(\Xi(z))$ provided that
	the qualification condition
	$\partial^\infty\theta(\Xi(z))\cap\ker\Xi'(z)^\top=\{0\}$ holds,
	and the latter is implied by item~\ref{item:dec_injectivity}:
	for each $q\in\dom\theta$, we have
	$\partial^\infty\theta(q)
	=N_{\dom\theta}(q)\subset N_{\dom\theta}(0)\subset\HH$,
	where we used \cite[Proposition~8.12]{RockafellarWets1998},
	that $\dom\theta$ is a cone,
	and that $\mu+\xi\in\partial\theta(0)$ holds for all
	$\mu\in\partial\theta(0)$ and $\xi\in N_{\dom\theta}(0)$.
	Finally, any two multipliers associated with the same $y\in\partial g(z)$
	differ by an element of $\HH\cap\ker\Xi'(z)^\top$,
	see \eqref{eq:subdiff_sublinear},
	so validity of \eqref{eq:nondegeneracy} yields the claimed uniqueness.

	As for item~\ref{item:dec_multiplier_convergence}, first note that
	$\hat y\in\partial g(\bar z)$ follows from the closedness of
	$\gph(\partial g)$ postulated in \cref{ass:composite_problem_basic},
	so $\hat\mu:=\mu(\bar z,\hat y)$ is available.
	Due to $\mu(z_k,y_k),\hat\mu\in\partial\theta(0)$
	by \eqref{eq:subdiff_sublinear}, item~\ref{item:dec_injectivity} yields
	\[
		c\norm{\mu(z_k,y_k)-\hat\mu}
		\leq
		\nnorm{\Xi'(z_k)^\top(\mu(z_k,y_k)-\hat\mu)}
		=
		\nnorm{y_k-\Xi'(z_k)^\top\hat\mu}
		\to
		\nnorm{\hat y-\Xi'(\bar z)^\top\hat\mu}
		=
		0,
	\]
	proving the announced convergence $\mu(z_k,y_k) \to \mu(\bar z,\hat y)$.
\end{proof}

In particular, \cref{lem:decomposable_multipliers} applies to $(\bar z,\bar y)$
and provides the unique reduced multiplier
$\bar\mu:=\mu(\bar z,\bar y)\in\partial\theta(0)$ satisfying
$\bar y=\Xi'(\bar z)^\top\bar\mu$.
Let us fix the following pieces of notation
\[
	\KK
	:=
	\{q\in\R^p\,|\,\innerprod{\bar\mu}{q}=\theta(q)\},
	\quad
	\overline\Lambda
	:=
	\{\mu\in\partial\theta(0)\,|\,\nabla f_0(\bar x)+G'(\bar x)^\top\mu=0\},
\]
where $G:=\Xi\circ F$ is well defined and twice continuously differentiable
in a neighborhood of $\bar x$, and satisfies $G'(\bar x)=\Xi'(\bar z)F'(\bar x)$.
The set $\KK$ is called the critical cone of $\theta$ at $0$ for
$\bar\mu$. Since $\theta$ is the support function of the set $\partial\theta(0)$,
condition $\innerprod{\bar\mu}{q}=\theta(q)$ just means that $\bar\mu$ maximizes
$\innerprod{\cdot}{q}$ over $\partial\theta(0)$, so that we find the identities
$\KK=\{q\,|\,\forall\mu\in\partial\theta(0)\colon\,\innerprod{\mu-\bar\mu}{q}\leq 0\}
=N_{\partial\theta(0)}(\bar\mu)$
which recover the second representation of the subdifferential
in \cite[Corollary~8.25]{RockafellarWets1998}.
Consequently, $\KK$ is a closed, convex cone, and the bipolar theorem,
see e.g.\ \cite[Corollary~6.21]{RockafellarWets1998}, yields
\begin{equation}\label{eq:critical_cone_polar}
	\KK^\circ
	=
	\cl\cone\bigl(\partial\theta(0)-\bar\mu\bigr)
	\subset
	\HH.
\end{equation}
Furthermore, since, for $y=\Xi'(\bar z)^\top\mu$ with
$\mu\in\partial\theta(0)$, 
$\nabla_x\LL(\bar x,y)=\nabla f_0(\bar x)+G'(\bar x)^\top\mu$
is valid, \cref{lem:decomposable_multipliers} shows that
$\overline Y=\Xi'(\bar z)^\top\overline\Lambda$
is in one-to-one correspondence with the reduced multiplier set
$\overline\Lambda$, and $\bar\mu\in\overline\Lambda$.

The decisive second-order ingredient of our analysis is a formula for the
graphical derivative of $\partial g$ at $(\bar z,\bar y)$, which we borrow
from \cite[Theorem~A.1(b)]{HangSarabi2023}.
Behind it lies a remarkable feature of sublinearity.
By positive homogeneity, the second-order difference quotients of $\theta$
at $0$ for $\bar\mu$ satisfy
\[
	\forall q\in\R^p\colon\quad
	\frac{\theta(tq)-\theta(0)-t\innerprod{\bar\mu}{q}}{\tfrac12 t^2}
	=
	\frac{2\bigl(\theta(q)-\innerprod{\bar\mu}{q}\bigr)}{t}
\]
and, due to $\theta(\cdot)\geq\innerprod{\bar\mu}{\cdot}$, they increase
monotonically as $t\downarrow 0$.
Hence, they epi-converge, see \cite[Definition~7.1]{RockafellarWets1998}, 
to their pointwise supremum $\delta_{\KK}$,
see \cite[Proposition~7.4(d)]{RockafellarWets1998}, i.e.,
$\theta$ is twice epi-differentiable at $0$ for $\bar\mu$,
see \cite[Definition~13.6]{RockafellarWets1998}, 
with second subderivative $\delta_{\KK}$, 
see \cite[equation~(A.2)]{HangSarabi2023} and \cite[page~138]{Milzarek2016}.
Since $\theta$ is proper, lsc, and convex, 
\cite[Example~13.30, Theorem~13.40]{RockafellarWets1998}, thus, show
that $\partial\theta$ is proto-differentiable at $(0,\bar\mu)$,
see \cite[page~331]{RockafellarWets1998}, with
$D(\partial\theta)(0,\bar\mu)=\partial\bigl(\tfrac12\delta_{\KK}\bigr)=N_{\KK}$.
The chain rule from \cite[Theorem~A.1(b)]{HangSarabi2023},
which is available since the nondegeneracy condition \eqref{eq:nondegeneracy} implies,
in view of \eqref{eq:critical_cone_polar}, the weaker dual condition
$\KK^\circ\cap\ker\Xi'(\bar z)^\top=\{0\}$ imposed there,
transfers this formula to $\partial g$ and yields
\begin{equation}\label{eq:D_subdiff_decomposable}
	\forall \zeta\in\R^m\colon\quad
	D(\partial g)(\bar z,\bar y)(\zeta)
	=
	\nabla^2\innerprod{\bar\mu}{\Xi}(\bar z)\zeta
	+\Xi'(\bar z)^\top N_{\KK}(\Xi'(\bar z)\zeta).
\end{equation}

After all these preparations, let us state the announced characterization of
the primal-dual error bounds of interest 
in terms of noncriticality for $C^2$-decomposable functions.
Besides the standing assumptions of this subsection,
it relies on the following two conditions.
\begin{enumerate}[label=($\greeknum*$)]
	\item\label{item:ass_calmness_Yx} Mapping $Y_{\bar x}$ is calm at $((0,0),\bar y)$.
	\item\label{item:ass_closedness} The set
		\begin{equation}\label{eq:closedness_reduced}
			\SS 
			:=
			\KK^\circ-\bigl(\KK^\circ\cap\ker G'(\bar x)^\top\bigr)
		\end{equation}
		is closed.
\end{enumerate}
Both conditions are inherited from \cite{MordukhovichSarabi2019},
see \cref{sec:C2_cone_reducible} for the precise correspondence.
Let us at least mention that \ref{item:ass_closedness} is not overly restrictive.
Whenever $\KK$ is polyhedral, e.g., if $\theta$ is a polyhedral function,
$\SS$ is the sum of two polyhedral cones and, thus, closed,
which is in line with the observation from \cref{sec:gph_subdiff_polyhedral}
that no such condition is needed in polyhedral settings.
Furthermore, \ref{item:ass_closedness} holds trivially whenever
$\KK^\circ\cap\ker G'(\bar x)^\top=\{0\}$, and this is the case, exemplary, if $G'(\bar x)$ is surjective.
Finally, as already mentioned above, \cite[Theorem~2]{HangSarabi2023} shows
that conditions of this type can be avoided altogether when noncriticality of
$\bar y$ is replaced by a suitable second-order sufficient condition.

\begin{theorem}\label{thm:decomposable_characterization}
	In addition to the standing $C^2$-decomposability and nondegeneracy \eqref{eq:nondegeneracy},
	assume that \ref{item:ass_calmness_Yx} and \ref{item:ass_closedness}
	are satisfied.
	Then the following statements are equivalent.
	\begin{enumerate}
		\item The multiplier $\bar y$ is noncritical for $\bar x$.
		\item There are a neighborhood $\NN\subset\R^n\times\R^m$ of $(\bar x,\bar y)$
			and a constant $\kappa>0$ such that \eqref{eq:Isol_Calm_M_estimate_via_y}
			holds for all $(x,y)\in \NN$.
		\item There are a neighborhood $\NN\subset\R^n\times\R^m$ of $(\bar x,\bar y)$
			and a constant $\kappa>0$ such that \eqref{eq:Isol_Calm_M_estimate_via_y_special}
			holds for all $(x,y)\in \NN$.
	\end{enumerate}
\end{theorem}

In view of \cref{thm:full_characterization_of_error_bound},
the crucial step in the proof of \cref{thm:decomposable_characterization},
which we postpone for a moment,
is the verification of condition \ref{item:weak_inner_calmness},
which is the decisive result of this subsection.
The verification is driven by the characterization of weak inner calmness in
the fuzzy sense from \cref{Lem:weak_fuzzy_ic}.

\begin{proposition}\label{prop:weak_inner_calmness_decomposable}
	In addition to the standing $C^2$-decomposability and nondegeneracy \eqref{eq:nondegeneracy},
	assume that \ref{item:ass_calmness_Yx} and \ref{item:ass_closedness}
	are satisfied.
	Then $Y$ is weakly inner calm in the fuzzy sense at $((0,0,\bar x),\bar y)$
	in each direction $(0,0,d^x)$ such that $d^x\neq 0$, i.e.,
	condition \ref{item:weak_inner_calmness} of \cref{thm:full_characterization_of_error_bound} holds.
\end{proposition}
\begin{proof}
	Fix a direction $d^x\neq 0$ and assume that there are sequences $t_k\downarrow 0$,
	$y_k\to\bar y$, and $(d^v_k,d^u_k,d^x_k)\to(0,0,d^x)$
	such that $y_k\in Y(v_k,u_k,x_k)$ for each $k\in\N$,
	where $(v_k,u_k,x_k):=(t_kd^v_k,t_kd^u_k,\bar x+t_kd^x_k)$;
	otherwise, $Y$ is trivially weakly inner calm in the fuzzy sense
	at $((0,0,\bar x),\bar y)$ in direction $(0,0,d^x)$,
	see \cref{def:weak_fuzzy_inner_calmness}.
	Due to \cref{Lem:weak_fuzzy_ic}, it suffices to show
	$(0,0,d^x)\in\dom DY((0,0,\bar x),\bar y)$ and,
	by \cref{lem:graphical_derivative_Y}, the latter amounts to finding $d^y\in\R^m$ with
	\begin{equation}\label{eq:decomposable_criticality_witness}
		\nabla^2_{xx}\LL(\bar x,\bar y)d^x+F'(\bar x)^\top d^y=0,
		\quad
		d^y\in D(\partial g)(\bar z,\bar y)(F'(\bar x)d^x).
	\end{equation}
	Note that $v_k=t_kd^v_k=\oo(t_k)$ and $u_k=t_kd^u_k=\oo(t_k)$ due to
	$(d^v_k,d^u_k)\to(0,0)$, and let $C>0$ denote a generic constant
	which is enlarged when necessary.

	Let us start with some preparations.
	Set $w_k:=F(x_k)+u_k$ and note that $\norm{w_k-\bar z}\leq Ct_k$, so $w_k\to\bar z$.
	By definition of $Y$, we have $\nabla_x\LL(x_k,y_k)=v_k$ and
	$y_k\in\partial g(w_k)$ for each $k\in\N$.
	Hence, for each sufficiently large $k\in\N$,
	\cref{lem:decomposable_multipliers} provides the unique multiplier
	$\mu_k:=\mu(w_k,y_k)\in\partial\theta(s_k)$ satisfying
	$y_k=\Xi'(w_k)^\top\mu_k$, where $s_k:=\Xi(w_k)$,
	and guarantees $\mu_k\to\bar\mu$.
	By \eqref{eq:subdiff_sublinear}, the multiplier relation means nothing
	else but $\mu_k\in\partial\theta(0)$ and
	$\innerprod{\mu_k}{s_k}=\theta(s_k)$.
	Furthermore, a first-order Taylor expansion of $\Xi$ at $\bar z$,
	exploiting $\Xi(\bar z)=0$, gives $q_k:=s_k/t_k\to G'(\bar x)d^x=:\bar q$,
	and the lower semicontinuity and positive homogeneity of $\theta$ yield
	\begin{equation}\label{eq:qbar_critical}
		\theta(\bar q)
		\leq
		\liminf_{k\to\infty}\theta(q_k)
		=
		\liminf_{k\to\infty}\innerprod{\mu_k}{q_k}
		=
		\innerprod{\bar\mu}{\bar q}
		\leq
		\theta(\bar q),
	\end{equation}
	where the last estimate is due to $\bar\mu\in\partial\theta(0)$.
	Hence, $\theta(\bar q)=\innerprod{\bar\mu}{\bar q}$, i.e., $\bar q\in\KK$.

	Next, we utilize the calmness of $Y_{\bar x}$ imposed in \ref{item:ass_calmness_Yx}.
	Arguing as in the proof of \cref{lem:semi_isolated_calmness_as_error_bound},
	with $x:=\bar x$ fixed, it can be equivalently recast as the error bound
	\[
		\dist(y,\overline Y)
		\leq
		\rho\bigl(\norm{\nabla_x\LL(\bar x,y)}+\dist(\bar z,(\partial g)^{-1}(y))\bigr)
	\]
	for all $y$ sufficiently close to $\bar y$ and some constant $\rho>0$.
	For $y:=y_k$, $k\in\N$, we can estimate
	$\norm{\nabla_x\LL(\bar x,y_k)}\leq\norm{v_k}+C\norm{x_k-\bar x}\leq Ct_k$
	by local Lipschitz continuity of $\nabla_x\LL$ w.r.t.\ $x$, 
	uniformly along the bounded sequence $y_k$,
	as well as $\dist(\bar z,(\partial g)^{-1}(y_k))\leq\norm{w_k-\bar z}\leq Ct_k$
	due to $w_k\in(\partial g)^{-1}(y_k)$.
	For each $k\in\N$, this yields a multiplier $y'_k\in\overline Y$ with $\norm{y_k-y'_k}\leq Ct_k$,
	and, clearly, $y'_k\to\bar y$.
	Due to $\overline Y=\Xi'(\bar z)^\top\overline\Lambda$,
	\cref{lem:decomposable_multipliers} provides the associated reduced
	multipliers $\mu'_k:=\mu(\bar z,y'_k)\in\overline\Lambda$ which satisfy
	$y'_k=\Xi'(\bar z)^\top\mu'_k$ and converge to $\bar\mu$.
	Since $\mu_k-\mu'_k\in\partial\theta(0)-\partial\theta(0)\subset\HH$,
	\cref{lem:decomposable_multipliers} also gives
	\[
		c\nnorm{\mu_k-\mu'_k}
		\leq
		\nnorm{\Xi'(\bar z)^\top(\mu_k-\mu'_k)}
		=
		\nnorm{y_k-y'_k-(\Xi'(w_k)-\Xi'(\bar z))^\top\mu_k}
		\leq
		Ct_k
	\]
	due to the local Lipschitz continuity of $\Xi'$ and boundedness of $\mu_k$.
	Thus, the sequence given by $\eta_k:=(\mu_k-\mu'_k)/t_k$, $k\in\N$, is bounded and converges,
	along a subsequence (without relabeling), to some $\bar\eta\in\R^p$.

	The decisive step is to locate $\bar\eta$.
	Set $\QQ:=-\bigl(\KK^\circ\cap\ker G'(\bar x)^\top\bigr)$
	and note that $\QQ$ is a closed, convex cone satisfying
	$\SS =\KK^\circ+\QQ$, see \eqref{eq:closedness_reduced}.
	Due to $\mu_k\in\partial\theta(0)$ and
	$\bar\mu,\mu'_k\in\overline\Lambda\subset\partial\theta(0)$,
	the polarity relation \eqref{eq:critical_cone_polar} gives
	\[
		\mu_k-\mu'_k
		=
		(\mu_k-\bar\mu)-(\mu'_k-\bar\mu)
		\in
		\KK^\circ
		-
		\bigl(\KK^\circ\cap\ker G'(\bar x)^\top\bigr)
		=
		\KK^\circ+\QQ
		=
		\SS ,
	\]
	where we exploited $G'(\bar x)^\top(\mu'_k-\bar\mu)=0$,
	valid by definition of $\overline\Lambda$.
	As all the involved sets are cones, this gives $\eta_k\in\SS $ for each $k\in\N$,
	and the closedness of $\SS $, guaranteed by \ref{item:ass_closedness},
	yields $\bar\eta\in\SS $.
	Let us now verify
	\begin{equation}\label{eq:eta_in_normal_cone}
		\bar q \in \KK\cap\QQ^\circ,
		\quad
		\innerprod{\bar\eta}{\bar q} = 0.
	\end{equation}
	Note that we have 
	\[
		(\KK\cap\QQ^\circ)^\circ
		=
		\cl(\KK^\circ + (\QQ^\circ)^\circ)
		=
		\cl(\KK^\circ+\QQ)
		=
		\cl\SS
		=
		\SS
	\]
	according to \cite[equation (2.32)]{BonnansShapiro2000},
	\cite[Corollary~6.21]{RockafellarWets1998}, and \ref{item:ass_closedness}.
	Hence, $\innerprod{\bar\eta}{q}\leq 0$ holds 	for all $q\in\KK\cap\QQ^\circ$.
	Moreover, $\bar q\in\KK$ has been shown via \eqref{eq:qbar_critical},
	and $\bar q=G'(\bar x)d^x$ gives
	$\innerprod{\bar q}{s}=\innerprod{d^x}{G'(\bar x)^\top s}=0$ for each $s\in\QQ$,
	i.e., $\bar q\in\QQ^\circ$.
	Finally, the monotonicity of the subdifferential mapping $\partial\theta$
	associated with the convex function $\theta$,
	see \cite[Theorem~12.17]{RockafellarWets1998}, gives
	$0\leq\innerprod{\mu_k-\mu'_k}{s_k-0}=t_k^2\innerprod{\eta_k}{q_k}$
	due to $\mu_k\in\partial\theta(s_k)$ and $\mu'_k\in\partial\theta(0)$,
	so taking the limit $k\to\infty$ yields $\innerprod{\bar\eta}{\bar q}\geq 0$.
	Taking everything together, \eqref{eq:eta_in_normal_cone} follows.
	Due to $\SS=\KK^\circ+\QQ$, we find $\eta\in\KK^\circ$ and $\eta'\in\QQ$ 
	such that $\bar\eta=\eta+\eta'$. With the aid of \eqref{eq:eta_in_normal_cone},
	$0 \leq 	\ninnerprod{-\eta}{\bar q} = \ninnerprod{\eta'-\bar\eta}{\bar q}
	= \ninnerprod{\eta'}{\bar q} \leq 0$
	follows, yielding $\innerprod{\eta}{\bar q}=0$ and, thus,
	$\eta\in N_\KK(\bar q)$ since $\KK$ is a closed, convex cone.
	Furthermore, $\eta'\in\QQ\subset\ker G'(\bar x)^\top$ holds true.
	Hence, we end up with
	\begin{equation}\label{eq:eta_located}
		\bar\eta=\eta+\eta',
		\quad
		\eta\in N_{\KK}(\bar q),
		\quad
		G'(\bar x)^\top\eta'=0.
	\end{equation}

	Let us now take care of the first relation in \eqref{eq:decomposable_criticality_witness}.
	Using $\norm{w_k-F(x_k)}=\norm{u_k}=\oo(t_k)$, boundedness of $x_k$ and $\mu_k$,
	continuity of $F$, and local Lipschitz continuity of $\Xi'$, we obtain
	\[
		\oo(t_k)
		=
		v_k
		=
		\nabla f_0(x_k)+F'(x_k)^\top\Xi'(w_k)^\top\mu_k
		=
		\nabla f_0(x_k)+G'(x_k)^\top\mu_k+\oo(t_k).
	\]
	Due to $\nabla f_0(\bar x)+G'(\bar x)^\top\bar\mu=0$, a first-order Taylor expansion of
	$x\mapsto\nabla f_0(x)+G'(x)^\top\bar\mu$ at $\bar x$ yields
	\begin{align*}
		\nabla f_0(x_k)+G'(x_k)^\top\mu_k
		&=
		\nabla f_0(x_k)+G'(x_k)^\top\bar\mu
		+G'(x_k)^\top(\mu_k-\mu'_k)
		+G'(x_k)^\top(\mu'_k-\bar\mu)
		\\
		&=
		t_kHd^x_k
		+t_kG'(\bar x)^\top\eta_k
		+\oo(t_k),
	\end{align*}
	where $H:=\nabla^2f_0(\bar x)+\nabla^2\innerprod{\bar\mu}{G}(\bar x)$,
	and where we used $G'(\bar x)^\top(\mu'_k-\bar\mu)=0$ together with the estimates
	$\nnorm{(G'(x_k)-G'(\bar x))^\top(\mu_k-\mu'_k)}\leq Ct_k\cdot t_k\nnorm{\eta_k}=\oo(t_k)$ 
	as well as
	$\nnorm{(G'(x_k)-G'(\bar x))^\top(\mu'_k-\bar\mu)}\leq Ct_k\nnorm{\mu'_k-\bar\mu}=\oo(t_k)$.
	Dividing by $t_k$ and passing to the limit gives $Hd^x+G'(\bar x)^\top\bar\eta=0$
	and, in view of \eqref{eq:eta_located},
	\begin{equation}\label{eq:reduced_witness_equation}
		Hd^x+G'(\bar x)^\top\eta=0.
	\end{equation}

	It remains to transfer these findings back to the original data.
	A direct computation based on the chain rule and $\bar y=\Xi'(\bar z)^\top\bar\mu$ reveals
	\[
		H=\nabla^2_{xx}\LL(\bar x,\bar y)+F'(\bar x)^\top\nabla^2\innerprod{\bar\mu}{\Xi}(\bar z)F'(\bar x).
	\]
	Hence, setting
	$d^y:=\nabla^2\innerprod{\bar\mu}{\Xi}(\bar z)F'(\bar x)d^x+\Xi'(\bar z)^\top\eta$
	and observing $\Xi'(\bar z)F'(\bar x)d^x=G'(\bar x)d^x=\bar q$,
	we obtain $d^y\in D(\partial g)(\bar z,\bar y)(F'(\bar x)d^x)$
	from \eqref{eq:D_subdiff_decomposable} and \eqref{eq:eta_located}
	as well as
	\[
		\nabla^2_{xx}\LL(\bar x,\bar y)d^x+F'(\bar x)^\top d^y
		=
		Hd^x+G'(\bar x)^\top\eta
		=
		0
	\]
	from \eqref{eq:reduced_witness_equation}.
	This verifies \eqref{eq:decomposable_criticality_witness} and completes the proof.
\end{proof}

\begin{proof}[Proof of \cref{thm:decomposable_characterization}]
	Condition \ref{item:weak_inner_calmness} of
	\cref{thm:full_characterization_of_error_bound} is valid by
	\cref{prop:weak_inner_calmness_decomposable},
	while condition \ref{item:calmness} follows from \ref{item:ass_calmness_Yx}
	via \cref{cor:calmness_Y_equivalence_composite}.
	Hence, analogously to the proof of \cref{thm:gph_of_subdifferential_polyhedral},
	the claim follows by combining \cref{thm:full_characterization_of_error_bound}
	and \cref{lem:semi_isolated_calmness_as_error_bound}.
\end{proof}

At the first glance, condition \ref{item:ass_closedness} seems to depend on the
decomposition \eqref{eq:decomposition} we have fixed, since the critical cone
$\KK$ is defined in terms of $\theta$ and $\bar\mu$
while function $G$ is the composition of $\Xi$ and $F$.
The following remark clarifies that this is not the case.

\begin{remark}\label{rem:closedness_representation_free}
	Evaluating \eqref{eq:D_subdiff_decomposable} at $\zeta:=0$ yields
	\[
		\Sigma
		:=
		D(\partial g)(\bar z,\bar y)(0)
		=
		\Xi'(\bar z)^\top N_{\KK}(0)
		=
		\Xi'(\bar z)^\top\KK^\circ,
	\]
	and the set on the left does not refer to \eqref{eq:decomposition} at all.
	Due to the nondegeneracy condition \eqref{eq:nondegeneracy},
	the linear mapping $L\colon\HH\to\R^m$ given by
	$L\nu:=\Xi'(\bar z)^\top\nu$ is injective,
	so it is a homeomorphism from $\HH$ onto the subspace $L(\HH)$.
	Furthermore, relation $G'(\bar x)^\top=F'(\bar x)^\top\Xi'(\bar z)^\top$ shows that
	some $\mu\in\KK^\circ$ satisfies $G'(\bar x)^\top\mu=0$
	if and only if $L\mu\in\ker F'(\bar x)^\top$ is valid,
	and, as $\SS\subset\HH$ holds by \eqref{eq:closedness_reduced}
	as well as $\KK^\circ\subset\HH$, see \eqref{eq:critical_cone_polar},
	we obtain
	\[
		L(\SS)
		=
		L(\KK^\circ)-L\bigl(\KK^\circ\cap\ker G'(\bar x)^\top\bigr)
		=
		\Sigma-\bigl(\Sigma\cap\ker F'(\bar x)^\top\bigr).
	\]
	Due to $\SS\subset\HH$ and the fact that $L$ is an homeomorphism
	from $\HH$ onto $L(\HH)$, closedness of $\SS$ and closedness of
	$L(\SS)$ are equivalent.
	Hence, \ref{item:ass_closedness} holds if and only if the set
	\begin{equation}\label{eq:closedness_original}
		D(\partial g)(\bar z,\bar y)(0)
		-
		\bigl(D(\partial g)(\bar z,\bar y)(0)\cap\ker F'(\bar x)^\top\bigr)
	\end{equation}
	is closed, i.e., \ref{item:ass_closedness} is a condition on the original
	data $g$, $F$, $\bar x$, and $\bar y$ alone, and, in particular,
	independent of the chosen decomposition \eqref{eq:decomposition}.
\end{remark}

\subsubsection{$C^2$-cone reducible constraints}\label{sec:C2_cone_reducible}

Let us now apply our findings to the setting of
\cite[Theorem~5.6]{MordukhovichSarabi2019},
where $g:=\delta_D$ is the indicator function of a closed set $D\subset\R^m$
which is so-called $C^2$-cone reducible at $\bar z$,
so that $\partial g=N_D$ is the limiting normal cone mapping.
The imposed reducibility means that there are an open neighborhood
$\VV\subset\R^m$ of $\bar z$, a twice continuously differentiable
mapping $h\colon\R^m\to\R^p$, and a closed, convex cone $K\subset\R^p$
such that $h(\bar z)=0$, $h'(\bar z)$ possesses full row rank $p$, and
\begin{equation}\label{eq:C2_reduction}
	D\cap\VV = \{z\in\VV\,|\,h(z)\in K\},
\end{equation}
see \cite[Definition~3.135]{BonnansShapiro2000} where this concept is
discussed in detail.
Exemplary, the second-order cone and the cone of all positive semidefinite matrices
are $C^2$-cone reducible, see \cite[Section~3.4.4]{BonnansShapiro2000}.

The reduction \eqref{eq:C2_reduction} yields precisely
$\delta_D(z)=\delta_K(h(z))$ for all $z\in\VV$, i.e.,
$g$ is $C^2$-decomposable at $\bar z:=F(\bar x)$ with the outer function
$\theta:=\delta_K$, which is proper, lsc, and sublinear
since $K$ is a closed, convex cone, and the inner mapping $\Xi:=h$.
Moreover, the full row rank of $h'(\bar z)$, i.e., the surjectivity of
$h'(\bar z)$, yields $\ker h'(\bar z)^\top=\{0\}$, so the nondegeneracy
condition \eqref{eq:nondegeneracy} holds trivially.
The data of \cref{sec:decomposable} specialize as follows:
we have $\partial\theta(0)=N_K(0)=K^\circ$,
the characterization \eqref{eq:subdiff_sublinear} reduces to the classical complementarity form
$\partial\theta(q)=N_K(q)=\{\mu\in K^\circ\,|\,\innerprod{\mu}{q}=0\}$ for $q\in K$, and
$\KK = \{q\,|\,\innerprod{\bar\mu}{q}=\delta_K(q)\} = K\cap\{\bar\mu\}^\perp$
is the critical cone to $K$ at $0$ for the unique reduced multiplier
$\bar\mu\in K^\circ$ satisfying $\bar y=h'(\bar z)^\top\bar\mu$.
Let us also mention that
\cref{lem:decomposable_multipliers}\,\ref{item:dec_chain_rule} recovers the
well-known representation $N_D(w)=h'(w)^\top N_K(h(w))$,
valid for all $w\in D$ sufficiently close to $\bar z$,
see \cite[Exercise~6.7]{RockafellarWets1998},
while \eqref{eq:D_subdiff_decomposable} recovers the graphical derivative
formula from \cite[equation (5.2)]{MordukhovichSarabi2019}.
Consequently, \cref{thm:decomposable_characterization} immediately yields the
following result.

\begin{corollary}\label{cor:C2_characterization}
	Let $g:=\delta_D$ be the indicator function of a closed set $D\subset\R^m$
	which is $C^2$-cone reducible at $\bar z$.
	Assume that $Y_{\bar x}$ is calm at $((0,0),\bar y)$ and that the set
	$\SS $ from \eqref{eq:closedness_reduced}, associated with the above
	specializations $\KK=K\cap\{\bar\mu\}^\perp$ and $G=h\circ F$,
	is closed.
	Then the following statements are equivalent.
	\begin{enumerate}
		\item The multiplier $\bar y$ is noncritical for $\bar x$.
		\item There are a neighborhood $\NN\subset\R^n\times\R^m$ of $(\bar x,\bar y)$
			and a constant $\kappa>0$ such that \eqref{eq:Isol_Calm_M_estimate_via_y}
			holds for all $(x,y)\in \NN$.
		\item There are a neighborhood $\NN\subset\R^n\times\R^m$ of $(\bar x,\bar y)$
			and a constant $\kappa>0$ such that \eqref{eq:Isol_Calm_M_estimate_via_y_special}
			holds for all $(x,y)\in \NN$.
	\end{enumerate}
\end{corollary}

Let us relate \cref{cor:C2_characterization} to
\cite[Theorem~5.6]{MordukhovichSarabi2019} in detail.
The assumptions imposed therein are the calmness of $Y_{\bar x}$
at $((0,0),\bar y)$ and the closedness of the set
\[
	\KK_D(\bar z,\bar y)^\circ
	-
	\bigl(\KK_D(\bar z,\bar y)^\circ\cap\ker F'(\bar x)^\top\bigr),
\]
where $\KK_D(\bar z,\bar y):=T_D(\bar z)\cap\{\bar y\}^\perp$
denotes the critical cone to $D$ at $\bar z$ for $\bar y$.
The first is precisely condition \ref{item:ass_calmness_Yx}.
As for the second one, the reduction \eqref{eq:C2_reduction} gives
$T_D(\bar z)=h'(\bar z)^{-1}(K)$ and, in view of $\bar y=h'(\bar z)^\top\bar\mu$,
this yields $\KK_D(\bar z,\bar y)=h'(\bar z)^{-1}(\KK)$.
Polarization, see \cite[Corollary~6.21]{RockafellarWets1998},
together with the injectivity of $h'(\bar z)^\top$,
which guarantees closedness of $h'(\bar z)^\top\KK^\circ$, thus, gives
\[
	\KK_D(\bar z,\bar y)^\circ
	=
	h'(\bar z)^\top\KK^\circ
	=
	D(\partial g)(\bar z,\bar y)(0),
\]
where the second equality has been observed in
\cref{rem:closedness_representation_free}.
In other words, the closedness assumption of
\cite[Theorem~5.6]{MordukhovichSarabi2019} is nothing else but the
representation-free form \eqref{eq:closedness_original} of condition
\ref{item:ass_closedness}, so that \cref{rem:closedness_representation_free}
renders \cite[Lemma~5.2]{MordukhovichSarabi2019} superfluous as well.
Hence, \cref{cor:C2_characterization} operates precisely under the
assumptions of \cite[Theorem~5.6]{MordukhovichSarabi2019} and recovers the
characterization of noncriticality obtained therein,
with semi-isolated calmness expressed as the error bound
\eqref{eq:Isol_Calm_M_estimate_via_y_special},
see \cref{lem:semi_isolated_calmness_as_error_bound},
enriched by the counterpart in terms of the error bound \eqref{eq:Isol_Calm_M_estimate_via_y}.

Let us close this section with a broader perspective.
\cref{thm:full_characterization_of_error_bound} identifies the precise mechanism
behind the connection of noncriticality and semi-isolated calmness in the form
of the two conditions \ref{item:weak_inner_calmness} and \ref{item:calmness},
which are not only sufficient (together with noncriticality) but also necessary ---
yet, being tailored exactly to this purpose, they may be difficult to verify as such.
The contribution of \cite{MordukhovichSarabi2019} can then be interpreted as
the identification of a rich and practically important class of problems,
governed by $C^2$-cone reducible constraint systems,
together with verifiable assumptions which secure precisely these two conditions,
and, as demonstrated in \cref{sec:decomposable} by means of the second-order
analysis from \cite{HangSarabi2023}, this mechanism extends to the
substantially larger class of $C^2$-decomposable functions.
In this sense, the approaches complement each other:
the abstract analysis of \cref{sec:main} isolates the essence of the equivalence,
while the decomposable setting renders this essence operational,
constituting, besides the polyhedral setting of \cref{sec:gph_subdiff_polyhedral},
the second major class of problems where the conditions
\ref{item:weak_inner_calmness} and \ref{item:calmness} are verifiable.

\section{Conclusions}\label{sec:conclusions}

In this paper, we have characterized error bounds which are decisive for the
local convergence analysis of numerical optimization methods in terms of noncriticality of the
underlying multiplier.
The abstract viewpoint of \cref{sec:main}, where merely a set-valued mapping
assigning primal-dual pairs to a parameter is available, revealed that,
besides noncriticality, precisely two calmness-type properties of the associated multiplier
mapping are responsible for this equivalence, and that both of them are
indispensable.
Specifying this mechanism in composite optimization, we have shown that these
properties come for free in the polyhedral setting, while, beyond polyhedrality,
they are secured by the calmness of a restricted multiplier mapping together with
a closedness condition for the entire class of $C^2$-decomposable functions.

Let us point out some directions which seem to be worth pursuing.
First, the notion of noncriticality studied here is based on the graphical
derivative and, thus, lacks robustness w.r.t.\ perturbations of the reference point.
As discussed in \cref{sec:perturbations}, replacing the graphical derivative by one
of its robust counterparts gives rise to a whole family of stronger notions,
and their interrelations, as well as their potential to yield superlinear
convergence of Newton-type methods, remain to be explored.
In this regard, the semi-strict graphical derivative from \cite{HangSarabi2023}
seems to be a particularly promising candidate.
Second, since conditions \ref{item:weak_inner_calmness} and \ref{item:calmness}
are not only sufficient but also necessary, any further result which ties
noncriticality to error bounds has to secure precisely these two properties,
and it would be interesting to identify additional classes of functions $g$,
beyond the ones discussed in this paper,
where this can be accomplished.
Finally, the interplay of the error bounds studied here with the convergence
analysis of augmented Lagrangian and sequential quadratic programming methods
in the general setting of \eqref{eq:composite_problem}
is a natural topic for future research.

\nocite{*}


\end{document}